\documentclass{article}

\usepackage{amsmath}
\usepackage{latexsym}
\usepackage{amsfonts}
\usepackage{amssymb}
\usepackage{amscd}
\usepackage{theorem}
\usepackage{a4wide}
\usepackage{color}
\usepackage[utf8]{inputenc}
\usepackage[title]{appendix}
\usepackage{manyfoot}

\theoremstyle{thmstyleone}
\newtheorem{theorem}{Theorem}
\newtheorem{lemma}[theorem]{Lemma}
\newtheorem{proposition}[theorem]{Proposition}
\newtheorem{corollary}[theorem]{Corollary}

\theoremstyle{thmstyletwo}%
\newtheorem{example}{Example}%
\newtheorem{remark}{Remark}%

\theoremstyle{thmstylethree}%
\newenvironment{proof}{    
	\noindent
	\textbf{Proof.}}{
	\hfill $\Box$
	\vspace{3mm}
}

\numberwithin{equation}{section}

\newcommand{\N}{\mathbb{N}} 
\newcommand{\R}{\mathbb{R}} 
\newcommand{\C}{\mathbb{C}} 
\newcommand{\D}{\mathbb{D}} 

\newcommand{\eps}{\varepsilon}

\DeclareMathOperator*{\ind}{ind}    

\newcommand{\vgd}{v_{\gamma+\delta}}
\newcommand{\vg}{v_{\gamma}}

\title{Cesàro operators associated with Borel measures acting on weighted spaces of holomorphic functions with sup-norms}

\author{María J. Beltrán-Meneu\footnote{Departament de Matemàtiques, Universitat Jaume I, mmeneu@uji.es}, José Bonet\footnote{Departament de Matemàtica Aplicada, Universitat Politècnica de València, jbonet@mat.upv.es} and Enrique Jordá\footnote{Departament de Matemàtica Aplicada, Universitat Politècnica de València, ejorda@mat.upv.es}}
\date{}

\begin{document}

\maketitle

%
%
%
%

\begin{abstract}

Let $\mu$ be a positive finite Borel measure on $[0,1).$ Cesàro-type operators $C_{\mu}$ when acting on weighted spaces of holomorphic functions are investigated. In the case of bounded holomorphic functions on the unit disc we prove that $C_\mu$ is continuous if and only if it is compact. In the case of weighted Banach spaces of holomorphic function defined by general weights, we give sufficient and necessary conditions for the continuity and compactness. For standard weights,  we    characterize the continuity and compactness on classical growth Banach spaces of holomorphic functions. We also study the  point spectrum and the spectrum of $C_\mu$ on  the space of  holomorphic functions on the disc, on the space of bounded holomorphic functions on the disc, and on the classical growth Banach spaces of holomorphic functions.  All  characterizations are given in terms of the sequence of moments  $(\mu_n)_{n\in\N_0}$.    The continuity, compactness and spectrum of $C_\mu$ acting on Fr\'echet and (LB) Korenblum type spaces are also considered.

\end{abstract}



	\section{Introduction}
Let $H(\D)$ be the space of holomorphic functions on the unit disc $\D$	of the complex plane $\C$ endowed with the topology of uniform convergence on the compact subsets of $\D$. If  $\mu$ is a positive finite Borel measure on $[0,1)$ and $f\in H(\D),$ $f(z)=\sum_{n=0}^{\infty}a_nz^n,$ $z\in \D,$ the Cesàro-type operator associated with $\mu$ is defined by
	 $$C_{\mu}(f)(z):=\sum_{n=0}^{\infty}\mu_n\left(\sum_{k=0}^{n}a_k\right)z^n, \ z\in \D,$$
	where $\mu_n:=\int_{0}^{1}t^nd\mu(t)$ is the {\it n-th moment of $\mu$}, $n\in \N_0:= \N\cup\{0\}$. This operator was studied by Galanopoulos et al. in \cite{GGM}. They proved, in particular, that $C_\mu:\ H(\D)\rightarrow H(\D)$  is continuous and has the integral representation (\cite[Proposition 1]{GGM}):
	$$ C_\mu(f)(z)=\int_{0}^{1}\frac{f(tz)}{1-tz}d\mu(t), \ f\in H(\D).$$
It is a generalization of the classical Cesàro operator $C$. Indeed, if $\mu$ is the Lebesgue measure, then 	$C_\mu=C,$ with  $\mu_n:=\int_{0}^{1}t^ndt=\frac{1}{n+1}, \ n\in \N.$

Several authors have analyzed the continuity and compactness of Cesàro-type operators on different Banach spaces of holomorphic functions on $\D.$   By the closed graph theorem, when dealing with Fréchet or (LB)-spaces $X$ continuously embedded in $H(\D),$ to prove the continuity of $C_\mu$  it suffices to show that the operator is well defined on $X$  \cite[Theorem 24.31]{meisevogt}. Galanopoulos et al. \cite{GGM} characterized the boundedness and compactness on the Hardy spaces $H^p,$  $1\leq p\leq \infty,$ on the weighted Bergman spaces $A_{\alpha}^p,$ $1< p<\infty$ and $\alpha>-1$, on {\it BMOA}, and on the Bloch space $\mathcal{B},$ obtaining conditions on the positive finite Borel measure $\mu.$ Measures of Carleson type played a basic role, relating boundedness  to Carleson measures in the case of $H^p$ and  $A_{\alpha}^p,$ and to logarithmic  Carleson measures in the case of {\it BMOA}, and $\mathcal{B}.$ Compactness was related to vanishing  and  vanishing logarithmic  Carleson measures, respectively.  In our study we deal with $s$-Carleson and vanishing $s$-Carleson measures.

Galanopoulos et al. \cite{GGM} also characterized the continuity of $C_{\mu}$  on the Banach space of bounded holomorphic functions on the disc $H^{\infty}(\D)$, space on which the Cesàro operator is not well defined \cite{DanikasSiskakis}, and asked whether $C_\mu(H^{\infty}(\D))\subseteq BMOA$. Blasco \cite{Blasco}  solved this question in the afirmative whenever $\mu$ is a positive Carleson measure and studied the boundedness of $C_{\mu}$ on Hardy, $BMOA$ and Bloch spaces for complex Borel measures $\mu.$ This question has also been solved by Bao et al. \cite{BaoSunWulan}. Recently, Galanopoulos et al. continue their study of the Cesàro-type operator on different Hilbert spaces of holomorphic functions \cite{GGM2023}, on the Dirichlet space and on analytic Besov spaces \cite{GGMM}.

In our work, we study the action of   Cesàro-like operators on weigthed spaces of holomorphic functions on $\D$ defined by sup-norms. A {\it weight} $v$ on $\D$  is  a   positive,  radial ($v(z)=v(|z|),\ z\in \D$), continuous and   decreasing function with respect to $|z|$ satisfying $\lim_{r\rightarrow 1^-}v(r)=0$. Given a weight $v$, we consider the {\it weighted Banach spaces of holomorphic funcions }
	
	$$H_v^{\infty}:=\{f\in H(\D): \ \sup_{z\in \D}v(z)|f(z)|<\infty\}$$
	and
	
	$$H_v^{0}:=\{f\in H_v^{\infty}: \ \lim_{|z|\rightarrow 1^-} v(z)|f(z)|=0\}.$$
	Both are Banach spaces under the norm $\|f\|_v:=\sup_{z\in \D}v(z)|f(z)|, \ f\in H_v^{\infty}.$ They appear in the study of growth conditions of analytic functions and have been investigated by several authors since the work of Shields and Williams (see eg.  \cite{BBG}, \cite{BBT}, \cite{Lusky2006}, \cite{SW1971} and the references therein). The space $H_v^{0}$ coincides with the closure of the polynomials in $H_v^{\infty}$ and $H_v^{\infty}$ is canonically isomorphic to  the bidual $(H_v^{0})^{**}$ \cite{BS}. Observe that the constant weight $v\equiv 1$ does not satisfy our assumptions. In this case, $H_v^{\infty}$ is the classical Banach space $H^\infty(\D)$ of bounded holomorphic functions on $\D$, and $H_v^{0}$ reduces to $\{ 0 \}$ in this case. We will state and prove some results for $H^\infty(\D)$, but when we refer to a weight in the statements, the constant weight is excluded, unless it is explicitly mentioned.  The so-called \textit{standard  weights} are defined by  $v_{\gamma}(r)=(1-r)^{\gamma}, $ $r>0,$ $\gamma>0,$ the spaces $H_{v_{\gamma}}^{\infty}$ and $H_{v_{\gamma}}^0$ are the classical growth Banach spaces $A^{-\gamma}$ and $A^{-\gamma}_0.$ These spaces play an important role in interpolation and   sampling of holomorphic functions \cite[Chapters 4 and 5]{Korenblum}.  Weighted Bloch spaces $\mathcal{B}_\gamma$ ($\mathcal{B}_{\gamma}^{0}$) are isomorphic to $H_{v_{\gamma}}^{\infty}\oplus \C$ ($H_{v_{\gamma}}^0\oplus \C$) by means of $f\mapsto (f',f(0))$. Shields and Williams \cite{SW1971} proved that  $H_{v_{\gamma}}^{\infty}$ is isomorphic to $\ell^{\infty}$ and that  $H_{v_{\gamma}}^0$ is isomorphic to $c_0$; see also \cite[Theorem 1.1]{Lusky2006}. Aleman and Persson  obtained the continuity of the classical operator $C$ on the spaces $H_{v_{\gamma}}^{\infty}$ and $H_{v_{\gamma}}^0$ and determined their spectrum (see \cite[Theorem 4.1]{Persson2008} and \cite[Theorems 4.1 and 5.1, Corollaries 2.1 and 5.1]{AlemanPersson2010}). 	Albanese et al. \cite{ABR2016} complemented their work  determining the norm of the operator and  proved that it  is not weakly compact on $H_{v_{\gamma}}^{\infty},$ neither on $H_{v_{\gamma}}^0.$

Given two functions  $f$  and $g$ defined in an (unbounded) interval of real numbers, we write   $f(x)=O(g(x))$ as $x\to a$ ($x\rightarrow \infty$) if there exist a constant  $C>0$  (and $M>0$) such that
$|f(x)|\leq C|g(x)|$ in a neighbourhood of $a$  ( for every $x\geq M$),
and $f(x)=o(g(x))$ as $x\to a$ if $\lim_{x\rightarrow a} \frac{f(x)}{g(x)}=0$ (here $a=\infty$ is allowed).  The expression $f(x) \lesssim g(x)$ means $f(x)=O(g(x)).$ When $f(x) \lesssim g(x)  \lesssim f(x),$ we write $f(x)\cong g(x).$ We recall that a finite positive Borel measure $\mu$ on $[0,1)$ is an {\it $s$-Carleson measure}, $s>0,$  if
$$ \mu([t,1))=O((1-t)^s), \ 0\leq t<1,$$
and a {\it vanishing $s$-Carleson measure} if
$$\mu([t,1))=o((1-t)^s) \text{ as $t\rightarrow 1^-$}.$$
 A  $1$-Carleson measure and a vanishing $1$-Carleson measure will be simply called a Carleson measure and a vanishing Carleson measure, respectively. By \cite[Lemma 2]{GGM}, $\mu$ is a Carleson measure if and only if $\mu_n=O(\frac{1}{n})$, $n\in \N,$ and $\mu$ is a vanishing Carleson measure if and only if $\mu_n=o(\frac{1}{n}).$

Our notation for concepts from functional analysis and operator theory is standard. We refer the reader e.g.\ to  \cite{Conway} and \cite{meisevogt}. 	 Given a Banach space $X,$ we denote by ${X}^*$ its topological dual, that is, the space of  all continuous linear forms on $X,$ and  given a continuous operator $T$ on $X,$  we denote by ${T}^*$ its \emph{transpose}.  For weighted Banach spaces of analytic functions of the type considered in this paper and operators between them, see \cite{pepesurvey}.

We briefly describe the content of the article.  In section \ref{sec:boundcompact} we study  the continuity and the compactness of $C_\mu$ on  different spaces of holomorphic functions on the unit disc. In the case of  $H^{\infty}(\D)$ we prove  in Theorem \ref{Cont_Hinfty} that $C_\mu$ is continuous if and only if it is compact, thus complementing \cite[Theorem 4]{GGM}. In the case of weighted Banach spaces of holomorphic function defined by general weights, we give sufficient and necessary conditions for the continuity and compactness  (Proposition \ref{weaklycompact} and Proposition \ref{Caract_cont_gral}). For standard weights,  we    characterize when $C_\mu$ is continuous  (Theorem \ref{Teorema cont}) and  compact  (Theorem \ref{Teorema compact}) on classical growth Banach spaces of holomorphic functions. When  we focus on operators whose domain and range is the same $H_{v_\gamma}^{\infty},$ we get:\\

\noindent 
{ {\bf Theorem}. Let $\gamma>0$ and let $\mu$ be a positive finite Borel measure on $[0,1)$. The following are equivalent:
\begin{itemize}
\item[(i)] $C_\mu: H_{v_\gamma}^{\infty}\to H_{v_\gamma}^{\infty}$ is a well defined continuous (compact) operator.
\item[(ii)]  $\mu$ is a (vanishing) Carleson measure.
\end{itemize}

 In section \ref{sec:Spectrum}  we analyse the spectrum and the point spectrum of $C_\mu$ on $H(\D),$ on $H^{\infty}(\D),$ and on the classical growth Banach spaces of holomorphic functions. All  characterizations are given in terms of the sequence of moments  $(\mu_n)_{n\in\N_0}$.  The continuity, compactness and spectrum of $C_\mu$ acting on Fr\'echet and (LB) Korenblum type spaces are considered in Section \ref{F-LB_Korenblum}.   We refer to \cite[Chapter 25]{meisevogt} for the definitions of Fr\'echet spaces and (LB) spaces.  Finally, we include  an Appendix devoted to give concrete examples of sequences $(\mu_n)_{n\in\N_0}$ which are solutions of  the Hausdorff moment problem, that is, sequences $(\mu_n)_{n\in\N_0}$ which are the moments of a positive finite Borel measure $\mu$ supported on $[0,1)$. These examples play a role in Sections \ref{sec:boundcompact} and \ref{sec:Spectrum}.

\section{Boundedness and compactness of $C_{\mu}$ on spaces of holomorphic functions}\label{sec:boundcompact}

\subsection{Weighted Banach spaces of holomorphic functions}

In this section we study the  Cesàro-type operator $C_\mu$ on weighted Banach spaces of holomorphic functions defined by general weights as well as for $H^\infty(\D)$.

In the proof of next Lemma we use the following notation: for  a weight $v$, $B^\infty_v$ and $B^0_v$ denote the closed unit ball of $H_{v}^\infty$ and $H_{v}^0$, respectively. It is easy to see by Montel's theorem that $B^\infty_v$ is a compact subset of $H(\D)$ for the topology $\tau_0$ of uniform convergence on the compact subsets of $\D$. Moreover, \cite[Lemma 1.1]{BBT} implies that $B^0_v$ is $\tau_0$-dense in $B^\infty_v$ and \cite[Theorem 1.1 and Example 2.1]{BS} yields that $H^\infty_v$  and $(H^0_v)^{**}$ are isometrically isomorphic. In fact, the compact open topology $\tau_0$, the pointwise topology and the weak* topology $\sigma(H^\infty_v,(H^0_v)^*)$ coincide on the closed unit ball $B^\infty_v$, by Alaoglu-Bourbaki's theorem, and consequently on all the bounded subsets.

\begin{theorem}\label{lemma_cont_Hv_Hv0}
		Let $\mu$ be a positive finite Borel measure on $[0,1)$ and $v, w$ be weights on $\D$. The following conditions are equivalent:
		\begin{itemize}
			\item[(i)] $C_\mu: H_{v}^\infty\to H_{w}^\infty$ is continuous and $C_\mu(1)=\sum_{n=0}^{\infty}\mu_n z^n\in H_{w}^0 $.
			\item[(ii)] $C_\mu: H_{v}^0\to H_{w}^0$ is continuous.
		\end{itemize}
		If these equivalent conditions hold, then the bitranspose $C_{\mu}^{**}$ of $C_\mu: H_{v}^0\to H_{w}^0$ coincides with $C_\mu: H_{v}^\infty\to H_{w}^\infty,$ and $\|C_\mu\|_{\mathcal{L}(H_{v}^\infty, H_{v}^\infty)}=\|C_\mu\|_{\mathcal{L}(H_{v}^0, H_{v}^0)}.$
		
	\end{theorem}
	
	\begin{proof}
		(i)$\Rightarrow$ (ii)  If  $C_\mu(1)=\sum_{n=0}^{\infty}\mu_n z^n\in H_{w}^0,$  then $C_\mu(z^N)\in H_{w}^0$ for every $N\in \N.$ Indeed, $C_\mu(z^N)= C_\mu(1)-\sum_{n=0}^{N-1}\mu_nz^n\in H_{w}^0$, as  the polynomials $\mathcal{P}$ are contained in $H_{w}^0.$ Moreover, since the set of  polynomials $\mathcal{P}$ is densely contained in $H_{w}^0,$  $C_{\mu}(H_{w}^0) =C_{\mu}(\overline{\mathcal{P}})\subseteq \overline{C_{\mu}(\mathcal{P})}\subseteq H_{w}^0,$ where the closure is taken with respect to $\|\ \|_w.$ So, the closed graph theorem yields the continuity.\\
		\noindent
		(ii)$\Rightarrow$ (i) If $C_\mu: H_{v}^0\to H_{w}^0$ is continuous, then it is obvious that $C_\mu(1)\in H_{w}^0.$  Moreover, there is $M>0$ such that $C_\mu(B^0_v) \subseteq M B^0_w$. The continuity of $C_\mu:H(\D) \rightarrow H(\D)$ implies (with closures taken in $(H(\D), \tau_0)$)
$$
C_\mu(B^\infty_v) = C_\mu(\overline{B^0_v}) \subseteq \overline{C_\mu(B^0_v)} \subseteq M \overline{B^0_w} \subseteq M B^\infty_w.
$$
This implies that $C_\mu: H_{v}^\infty\to H_{w}^\infty$ is well defined and continuous.

Now assume that these two equivalent conditions are satisfied. By our comments before the Lemma (see \cite{BS}), $C_\mu^{**}: H_{v}^\infty\to H_{w}^\infty$ and $C_\mu: H_{v}^\infty\to H_{w}^\infty$ are two continuous linear operators which coincide on $H^0_v$. To see that they also coincide on $H^\infty_v$, take $f \in B^\infty_v$. There is a sequence $(f_n)_n \in B^0_v$ which converges to $f$ for the topology $\tau_0$. Since $C_\mu: H(\D) \rightarrow H(\D)$ is continuous, $(C_\mu(f_n))_n$ converges to $C_\mu(f)$ in $(H(\D),\tau_0)$. On the other hand, $C_{\mu}^{**}$ is weak*-weak* continuous and $(f_n)_n$ converges to $f$ for the topology $\sigma(H^\infty_v,(H^0_v)^*)$. Hence, the sequence $C_\mu(f_n)=C^{**}_{\mu}(f_n), n=1,2,...$ converges to $C^{**}_{\mu}(f)$ in $H^\infty_w$ for the topology $\sigma(H^\infty_w,(H^0_w)^*)$. Since this topology coincides with $\tau_0$ on the bounded subsets of $H^\infty_w$, we get $C^{**}_{\mu}(f) = C_\mu(f)$. The coincidence that we have just proved implies the assertion about the operator norms.
	\end{proof}
	
	\begin{proposition}\label{exemples C_mu(1)}
	Let $\mu$ be a positive finite Borel measure on $[0,1).$  The following hold:
	\begin{itemize}
	\item[(i)] If  $\mu$ is a Carleson measure and $w$ a weight such that  $w(r)|\log (1-r)|\rightarrow 0$ as $r\rightarrow 1^-$,  then $C_{\mu}(1)\in H_{w}^{0}.$ This condition is satisfied, for instance, by the standard weights.
	\item[(ii)] 	If $\mu$ is an $s$-Carleson measure for $0<s<1$, then $C_\mu(1)\in H_{v_\gamma}^{0}$ for each $\gamma>1-s$. 
	\end{itemize}	
	
	\end{proposition}

\begin{proof}
(i) By assumption, there exists $C>0 $ such that $\mu_n\leq \frac{C}{n+1}$ for every $n\in \N$. So,  for $z\in \D, z\neq 0,$
$$w(z)|C_{\mu}(1)(z)|\leq w(z)\sum_{n=0}^{\infty}\mu_n|z|^n\leq Cw(z) \sum_{n=0}^{\infty}\frac{|z|^n}{n+1}=-Cw(z)\frac{\log(1-|z|)}{|z|}.$$
Now, the hypothesis on $w$  implies   $\lim_{|z|\rightarrow 1^-}w(z)|C_{\mu}(1)(z)|=0.$ \\
	
(ii) Fix $\gamma>1-s$ and  take $p\in \R$ such that $p>1/s$ and $(p-1)/p<\gamma$. This choice is possible since $\lim_{p\to \frac1s}\frac{p-1}{p}=1-s$.  Let $q:=p/(p-1)$. Then $(\mu_n)_n\in l_p$, $(z^n)_n\in l_q$ for any $z\in\D$  and $\|(z_n)_n\|_q=\frac{1}{(1-|z|^q)^{\frac1q}}$. We apply H\"older inequality to get
$$|C_\mu(1)(z)|\leq \sum_{n=0}^{\infty} \mu_n |z|^n\leq \|(\mu_n)_n\|_p\frac{1}{(1-|z|^q)^{\frac1q}}.$$
\noindent Since $\frac1q=\frac{p-1}{p}<\gamma$, this yields $|C_\mu(1)(z)|=o\left(\frac{1}{(1-|z|)^\gamma}\right)$ as $|z|\rightarrow 1^-$.  
	
\end{proof}

\begin{example}
Let $\mu$ be a positive finite Borel measure on $[0,1)$ with $\sum_{n=0}^\infty \mu_n = \infty$. Then $v(z):= 1/\sum_{n=0}^\infty \mu_n |z|^n, \ z \in \D,$ is a weight on $\D$. The function $f(z):= \sum_{n=0}^\infty \mu_n z^n = C_\mu(1)(z), \ z \in \D,$ belongs to $H(\D)$ and has radius of convergence $1$. Moreover, $f \in H_v^\infty \setminus H_v^0$, since $\max_{|z|=r} |f(z)| = \sum_{n=0}^\infty \mu_n r^n$ for each $0 \leq r < 1$. On the other hand, $C_{\mu}: H^\infty(\D) \rightarrow  H_v^\infty$ is continuous, because, for each $f \in  H^\infty(\D)$ and $z \in \D$, we have
$$
|C_\mu(f)(z)| \leq ||f||_\infty \int_0^1 \frac{d \mu(t)}{|1-tz|} \leq ||f||_\infty \int_0^1 (\sum_{n=0}^\infty t^n |z|^n) d \mu(t) \leq ||f||_\infty \frac{1}{v(|z|)}.
$$	
\end{example}

	\begin{proposition}\label{weaklycompact}
Let $\mu$ be a positive finite Borel measure on $[0,1)$ and $v, w$ be weights on $\D$ such that $C_\mu: H_{v}^0\to H_{w}^0$ is continuous. Then
		$C_\mu: H_{v}^0\to H_{w}^0$ and	$C_\mu: H_{v}^{\infty}\to H_{w}^{\infty}$  are  compact if and only if $C_\mu(H_v^\infty)\subseteq H_w^0.$
	\end{proposition}	
	\begin{proof}
	By   Theorem \ref{lemma_cont_Hv_Hv0}, the operators $C_{\mu}^{**}$ and $C_\mu$ coincide on $H_{v}^\infty$. Therefore $C_\mu(H_v^\infty)\subseteq H_w^0$ if and only if $C^{**}_{\mu}((H_v^0)^{**})\subseteq H_w^0$. This fact holds if and only if $C_\mu: H_{v}^0\to H_{w}^0$ is weakly compact by
 \cite[Proposition 17.2.7]{Jarchow}. Now \cite[Corollary 2.5]{Alex} implies that this is equivalent to the compactness of $C_\mu: H_{v}^0\to H_{w}^0$. Finally,  Theorem \ref{lemma_cont_Hv_Hv0} also yield that $C_\mu: H_{v}^{\infty}\to H_{w}^{\infty}$ is compact if and only if $C_\mu: H_{v}^0\to H_{w}^0$ is compact.
	\end{proof}

	\begin{proposition}\label{Caract_cont_gral}
		Let $\mu$ be a positive finite Borel measure on $[0,1)$ and $v, w$ be weights on $\D.$
		\begin{itemize}
			\item[(i)] If
			$C(v,w):=\sup_{0\leq r<1}w(r)\int_{0}^{1}\frac{d\mu(t)}{v(tr)(1-tr)}<\infty,$
			then  $C_\mu: H_{v}^\infty\to H_{w}^\infty$ is continuous.
			\item[(ii)] If  $\lim_{r\to 1^-}w(r)\int_0^1\frac{d\mu(t)}{v(tr)(1-tr)}=0,$
			then $C_\mu(H_v^\infty)\subseteq H_w^0$. In particular, the condition implies that $C_\mu: H_{v}^{\infty}\to H_{w}^0$ and $C_\mu: H_{v}^0\to H_{w}^0$ are  compact.
		\end{itemize}
		The conditions are also necessary if  the function $f:[0,1)\to \R_+$, $r\mapsto 1/v(r)$ admits holomorphic extension $\tilde{f}$ to $\D$ satisfying $|\tilde{f}(z)|\leq f(|z|)$ for every $z\in \D.$ In this case, $\|C_{\mu}\|= C(v,w).$
	\end{proposition}

	\begin{proof}
		Let $f\in H_v^{\infty}$ with $\|f\|_v\leq 1.$  For $0\leq r<1,$
		\begin{equation}\label{norm1}
			\sup_{z\in \D, |z|=r}w(z)\left|C_\mu f(z)\right|\leq \sup_{z\in \D, |z|=r}w(|z|)\int_0^1\frac{|f(tz)|}{(1-t|z|)}d\mu(t) \leq w(r)\int_0^1\frac{d\mu(t)}{v(tr)(1-tr)},
		\end{equation}
		so the condition in (i) yields the boundedness of $C_\mu: H_{v}^\infty\to H_{w}^\infty$  with $\|C_{\mu}\|\leq C(v,w),$ and the  condition in (ii), that $C_\mu(H_v^\infty)\subseteq H_w^0$, in particular, that $C_\mu(1)\in H_w^0.$ This fact, together with (i) and  Theorem \ref{lemma_cont_Hv_Hv0} implies the continuity of $C_\mu: H_{v}^0\to H_{w}^0.$  The  compactness of $C_\mu: H_{v}^0\to H_{w}^0$ and $C_\mu: H_{v}^{\infty}\to H_{w}^0$  follows by Proposition \ref{weaklycompact}.
		
		For the converses, assume  the function $f: [0,1)\to \R_+$, $r\mapsto 1/v(r)$ admits holomorphic extension $\tilde{f}$ to $\D$ satisfying $|\tilde{f}(z)|\leq f(|z|)$ for every $z\in \D.$ As $v(|z|)|\tilde{f}(z)|\leq v(|z|)f(|z|)$ and $v(r)f(r)=1,$ $0\leq r<1,$ we get $\|\tilde{f}\|_v=1.$ Moreover,
		\begin{equation}\label{norm2}
			w(r)\int_{0}^{1}\frac{d\mu(t)}{v(tr)(1-tr)}\leq  \sup_{z\in \D, |z|=r}w(z)\left|\int_0^1\frac{\tilde{f}(tz)}{(1-tz)}d\mu(t)     \right|= \sup_{z\in \D, |z|=r}w(z)|C_{\mu}(\tilde{f})|.
		\end{equation}
		So,  if we assume $C_{\mu}(\tilde{f})\in H_{v}^\infty,$ we get the condition in (i) is also necessary, and analogously for (ii) if  we assume $C_{\mu}(\tilde{f})\in H_{v}^0.$  The conclusion about the norm of the operator follows from  (\ref{norm1}) and  (\ref{norm2}).
	\end{proof}
	
	\begin{remark}
		Taking $\tilde{f}(z)=1/(1-z)^{\gamma}$, $z\in \D,$ we get that the necessity in  Proposition \ref{Caract_cont_gral} applies to standard weights $\vg(z)=(1-|z|)^{\gamma}$, $\gamma>0.$
	\end{remark}

	\begin{corollary}\label{Corollary_cont_gral}
		Let $\mu$ be a positive finite Borel measure on $[0,1)$ and $v, w$ be weights on $\D$ such that $w(|z|)\lesssim v(|z|)(1-|z|)^{\delta}, z\in \D.$
		\begin{itemize}
			\item [(i)]	If $\delta> 1,$ then $C_\mu: H_{v}^{\infty}\to H_{w}^0$ and $C_\mu: H_{v}^0\to H_{w}^0$ are continuous and compact.
			\item [(ii)]	If $\delta= 1,$ then $C_\mu: H_{v}^\infty\to H_{w}^\infty$ is continuous.
			\item[(iii)] If $0\leq\delta< 1,$ then:
			\begin{itemize}
				\item[(a)] If	$\sup_{0\leq r<1}v(r)(1-r)^{\delta}\int_{0}^{1}\frac{d\mu(t)}{v(tr)(1-tr)}<\infty,$ then $C_\mu: H_{v}^\infty\to H_{w}^\infty$ is continuous.
				\item[(b)] If	$\lim_{r\rightarrow 1^-}v(r)(1-r)^{\delta}\int_{0}^{1}\frac{d\mu(t)}{v(tr)(1-tr)}=0,$ then $C_\mu: H_{v}^{\infty}\to H_{w}^0$ and $C_\mu: H_{v}^0\to H_{w}^0$ are  compact.
			\end{itemize}

		\end{itemize}
		
	\end{corollary}
	
	\begin{remark}\label{sum_mu_n_dona_continu}
		Observe that the condition	in Corollary \ref{Corollary_cont_gral}(iii)(a) is satisfied, for instance, if $\int_{0}^{1}\frac{d\mu(t)}{(1-t)^{1-\delta}}<\infty,$ or equivalently, if $C_{\mu}((1-z)^{\delta})\in H^{\infty}(\D).$  If $\delta=0,$ the boundedness of the integral is equivalent to $\sum_n\mu_n<\infty$ and to $C_{\mu}(1)\in H^{\infty}(\D)$ (\cite[Theorem 4]{GGM}),  and thus,  this condition implies $C_{\mu}$ is continuous on $H_{v}^\infty$ and on  $H_{v}^0$ by  Theorem \ref{lemma_cont_Hv_Hv0}.
	\end{remark}
	

Now we see the  condition $\sum_n\mu_n<\infty$  in Remark \ref{sum_mu_n_dona_continu} also implies compactness of $C_\mu$  on $H_v^\infty$ and on $H_v^0$. It  yields the equivalence of the boundedness and compactness of $C_\mu$ on  $H^\infty(\D)$ for the case $v\equiv 1.$ 	 To do this, we need a technical lemma, which generalizes the {\em Weak Convergence Theorem} appearing in \cite[Section 2.4]{Shapiro}, widely used in the context of composition operators.   We recall that a Fr\'echet  Montel space is a complete metrizable locally convex space in which every closed and bounded set is compact (cf. \cite[Chapter 24]{meisevogt}).

\begin{lemma}
	\label{lemmacompact}
	Let $F$ be a Fr\'echet Montel space, and let $X,Y$ Banach spaces continuously embedded in $F$. Assume that $T:F\to F$ is continuous. If $T(X)\subseteq Y$ and $(T(f_n))_n$ is norm convergent to $0$ for each sequence $(f_n)_n$ bounded in $X$ and convergent to $0$ in $F$, then $T:X\to Y$ is a compact operator.
\end{lemma}	

\begin{proof}
	Let $T$ be as in the hypothesis. The closed graph theorem yields that $T:X\to Y$ is continuous. Assume  $T(B_X)$ is not precompact. Hence, there exists $\varepsilon>0$ and a sequence $(f_n)_n\subseteq B_X$  such that $\|T(f_n)-T(f_{m})\|>\eps$ for all $m\neq n$. Since $B_X$ is relatively compact in $F$, we can assume, by passing to a subsequence, that $(f_n)_n$ is convergent in $F$. Hence, $\frac12(f_n-f_{n+1})$ is convergent to 0 in $F$, and we get a contradiction with the assumption, as for each $n\in\N$, it holds
	$\left\|T\left(\frac12(f_n-f_{n+1})\right)\right\|>\frac\eps 2. $
\end{proof}

\begin{theorem}\label{sum_mu_n_implica_compact}
	Let $\mu$ be a positive finite Borel measure on $[0,1).$ If $\sum_{n=0}^{\infty}\mu_n<\infty,$ then $C_{\mu}: H_{v}^\infty \rightarrow H_{v}^\infty,$ $C_{\mu}: H_{v}^0\to H_{v}^0$ and  $C_\mu: H^\infty(\D)\rightarrow H^\infty(\D)$  are compact.	
\end{theorem}

\begin{proof}
	Let $v$ be a weight, including also the special case $v\equiv 1$ for which $H_{v}^\infty=H^\infty(\D).$ 		Let us assume without loss of generality that $\mu([0,1))=1$. By Remark \ref{sum_mu_n_dona_continu} and \cite[Theorem 4]{GGM},  $C_{\mu}$ is continuous on the Banach spaces considered, so by Lemma \ref{lemmacompact} we only need to show that $(\|C_\mu(f_n)\|_v)_n$ is convergent to 0 whenever $(f_n)_n\subseteq B_{H_v^\infty(\D)}$ is convergent to 0 in the compact open topology. Let $\eps>0$. The hypothesis yields  $\int_{0}^{1}\frac{d\mu(t)}{1-t}<\infty$ (\cite[Theorem 4]{GGM}). So, there is $s_0\in (0,1)$ such that
	\begin{equation}
		\label{p1}
		\int_{s_0}^{1}\frac{d\mu(t)}{1-t}<\frac{\eps}{2}.
	\end{equation}
	We have that
	\begin{equation}
		\label{p2}
		|C_\mu(f)(z)|=\left|\int_0^1\frac{f(tz)}{1-tz}d\mu(t)\right|
	\end{equation}
	\noindent  for all $f\in H_v^\infty$, $z\in \D$. For $z\in \D$, $t\in [s_0,1]$ and $f\in B_{H_v^\infty}$ we have $v(tz)|\frac{f(tz)}{1-tz} |\leq \frac{1}{1-t}$. Hence, by \eqref{p1}, for each $n\in\N$ we have
	
	\begin{equation}
		\label{p3}
		 v(z)\left|\int_{s_0}^1\frac{f_n(tz)}{1-tz}d\mu(t)\right|<\frac{\eps}{2}.
	\end{equation}
	Take $n_0\in \N$ such that, for every $n\geq n_0,$  $v(\omega)|f_n(\omega)|<\frac{\eps(1-s_0)}{2}$ for every $|\omega|\leq s_0.$	 For $z\in \D$ and $n\geq n_0$ we also have
	
	\begin{equation}
		\label{p4}
		 v(z)\left|\int_0^{s_0}\frac{f_n(tz)}{1-tz}d\mu(t)\right|\leq\int_0^{s_0}\frac{\max_{\omega\in s_0\overline{\D}} v(\omega)|f_n(\omega)|}{1-s_0}d\mu(t)< \frac{\eps}{2}.
	\end{equation}
	Putting together \eqref{p2}, \eqref{p3} and \eqref{p4} we conclude that $\|C_\mu(f_n)\|_v<\eps$ for $n\geq n_0$, and so, $C_\mu: H^\infty(\D)\rightarrow H^\infty(\D)$,   $C_{\mu}: H_{v}^\infty \rightarrow H_{v}^\infty$  and $C_{\mu}: H_{v}^0\to H_{v}^0$   are compact.	
\end{proof}

Theorem \ref{sum_mu_n_implica_compact} for $H^\infty(\D)$ permits us to conclude the following  theorem, which   shows that compactness is equivalent to continuity for $C_\mu: H^\infty(\D)\rightarrow H^\infty(\D)$. This continuity was characterized with several conditions in \cite[Theorem 4]{GGM} which we repeat in the statement.

\begin{theorem}\label{Cont_Hinfty}
	Let $\mu$ be a positive finite Borel measure on $[0,1).$ The following conditions are equivalent:
	\begin{itemize}
		\item[(i)] $C_\mu: H^\infty(\D)\rightarrow H^\infty(\D)$ is continuous.
		\item[(ii)] $C_\mu: H^\infty(\D)\rightarrow H^\infty(\D)$ is compact.
		\item[(iii)] $C_\mu(1)\in H^\infty(\D).$
		\item[(iv)] $\int_{0}^{1}\frac{d\mu(t)}{1-t}<\infty.$
		\item[(v)] $\sum_{n=0}^{\infty}\mu_n<\infty.$
	\end{itemize}
\end{theorem}

	Observe that  a positive finite Borel measure $\mu$  on $[0,1)$ such that $\sum_n \mu_n<\infty$ is  a vanishing Carleson measure by \cite[Theorem 3.3.1]{Knopp}.

	\begin{lemma}\label{PropNeces1}
		Let $\mu$ be a positive finite Borel measure on $[0,1)$ and $v, w$ weights on $\D.$ If   $f(z)=\sum_{n=0}^{\infty}a_nz^n\in H_v^{\infty}$ is such that $a_n\in \R, a_n>0$ for every $n\in \N_0,$ then  the following hold:
		\begin{itemize}
			\item[(i)]  If  $C_\mu(f)\in H_{w}^\infty,$ then there exists $C>0$ such that $$\sup_n \mu_n \|z^n\|_w\left(\sum_{k=0}^{n}\sum_{j=0}^{k}a_j\right)\leq C.$$
			\item[(ii)]  If  $C_\mu(f)\in H_{w}^0,$ then for every $\eps>0$ there exist $n_0\in \N$ and $r_0\in [0,1)$ such that
			$$\sup_{n\geq n_0}\sup_{r_0\leq r<1} \mu_n w(r)r^n\left(\sum_{k=0}^{n}\sum_{j=0}^{k}a_j\right)\leq \eps.$$
		\end{itemize}
	\end{lemma}
	
	\begin{proof}
		Let $f(z)=\sum_{k=0}^{\infty}a_kz^k\in H_v^{\infty}$ be as in the hypothesis and consider $C_{\mu}(f)=\sum_{k=0}^{\infty}\mu_k \left(\sum_{j=0}^{k} a_j\right)z^k\in H(\D).$ 	As  $(\mu_k)_k$ and $(r^k)_k,$ $0\leq r<1,$ are decreasing, 	given $n\in \N$ and $0\leq r_0<1$ we have
		\begin{eqnarray}
			&\sup_{r_0\leq |z|<1}& w(z)\left|\sum_{k=0}^{n}\mu_k\left(\sum_{j=0}^{k}a_j\right)z^k\right|= \sup_{r_0\leq r<1}w(r)\sum_{k=0}^{n}\mu_k\left(\sum_{j=0}^{k}a_j\right)r^k \\				
			&\geq & \sup_{r_0\leq r<1} \mu_n w(r)r^n\left(\sum_{k=0}^{n}\sum_{j=0}^{k}a_j\right). \nonumber
		\end{eqnarray}
		Thus, if $C_\mu(f)\in H_{w}^\infty,$ then the first supremum is bounded for $r_0=0$ and every $n\in \N.$ If $C_\mu(f)\in H_{w}^0,$ for every $\eps>0$ there exists $n_0\in \N$ and $r_0\in [0,1)$ such that it is bounded by $\eps$ for every $n\geq n_0.$
	\end{proof}

	\subsection{Classical growth Banach spaces of holomorphic functions}
	
	In the case of  weighted Banach spaces of holomorphic functions determined by  standard weights $\vg(r)=(1-r)^{\gamma},$ $0\leq r<1,\ \gamma>0,$ we  obtain a complete  characterization of the boundedness and compactness of the operator. The following theorem also provides a characterization of $s$-Carleson measures in terms of the boundedness of a family of integrals. The proof of (iii) $\Rightarrow$ (iv) was stated in \cite[Proposition 1]{CGP} without proof and a proof can be found in \cite[Corollary 4.6]{Blasco}. We include it for completeness.

	\begin{theorem}\label{Teorema cont}
		Let $\mu$ be a positive finite Borel measure on $[0,1)$  and $\gamma>0, \delta\in (-\gamma,1)$.  The following conditions are equivalent:
		\begin{itemize}
			\item[(i)] $C_\mu: H_{\vg}^\infty\rightarrow H_{\vgd}^\infty$ is a continuous operator
			\item[(ii)] $C_\mu\left(\frac{1}{(1-z)^{\gamma}}\right)\in H_{\vgd}^\infty$
			\item[(iii)] $\mu_n=O(\frac{1}{n^{1-\delta}})$
			\item[(iv)] $\mu $ is a $(1-\delta)$-Carleson measure
			\item[(v)] $\sup_{0\leq r<1}(1-r)^{\gamma+\delta} \int_{0}^{1}\frac{d\mu(t)}{(1-tr)^{\gamma+1}}<\infty$
			\item[(vi)] $C_\mu: H_{\vg}^0\rightarrow H_{\vgd}^0$ is a continuous operator.
		\end{itemize}
	\end{theorem}

	\begin{proof}
		Let us see  that the first five items are equivalent:\\
		\noindent
		(i)$\Rightarrow$ (ii) is trivial, as  $f(z)=\frac{1}{(1-z)^{\gamma}},\ z\in \D,$  clearly belongs to $H_{\vg}^\infty.$
		
		\noindent
		(ii)$\Rightarrow$ (iii)
		It is known   that $f(z)=\frac{1}{(1-z)^{\gamma}},\ z\in \D,$  satisfies $f(z)=\sum_{n=0}^{\infty}a_n(\gamma)z^n,\ z\in \D,$ with $a_n(\gamma) >0,$ $a_n(\gamma)  \cong n^{\gamma-1}$ (eg.\cite[p.12]{GGM}). Hence, if  $C_\mu\left(\frac{1}{(1-z)^{\gamma}}\right)\in H_{\vgd}^\infty,$   by Lemma \ref{PropNeces1} there exists $C>0$ such that
		$$\sup_n \mu_n \|z^n\|_{\vgd} \left(\sum_{k=1}^{n}\sum_{j=1}^{k}j^{\gamma-1}\right)\leq C.$$
		As $\|z^n\|_{\vgd}=\frac{(\gamma+\delta)^{\gamma+\delta}}{(n+\gamma+\delta)^{\gamma+\delta}}\frac{n^n}{(n+\gamma+\delta)^n},$ we get that, for every $n\in \N,$
		\begin{eqnarray}\label{serie1}
			 \mu_n\frac{(\gamma+\delta)^{\gamma+\delta}}{(n+\gamma+\delta)^{\gamma+\delta}}\frac{n^n}{(n+\gamma+\delta)^n}\left(\sum_{k=1}^{n}\sum_{j=1}^{k}j^{\gamma-1}\right)\leq  C
		\end{eqnarray}
		Observe that, if $\gamma\geq 1$ then $\sum_{j=1}^{m}j^{\gamma-1}\geq \int_{0}^{m}x^{\gamma-1}dx=\frac{m^{\gamma}}{\gamma},$ and if $0<\gamma<1,$ then $\sum_{j=1}^{m}j^{\gamma-1}\geq \int_{0}^{m}(x+1)^{\gamma-1}dx=\frac{(m+1)^{\gamma}-1}{\gamma},$  $m\in \N.$  So,   $\sum_{j=1}^{m}j^{\gamma-1}\gtrsim m^{\gamma} $ for every $\gamma>0.$
		Hence,
		$$ \sum_{k=1}^{n}\sum_{j=1}^{k}j^{\gamma-1}\gtrsim \sum_{k=1}^{n}  k^{\gamma}\gtrsim n^{\gamma +1}$$
		As a consequence, we get that  there exists $C'>0$ such that, for every $n\in \N,$
		
		 $$\mu_n\frac{n^{\gamma+1}}{(n+\gamma+\delta)^{\gamma+\delta}}\frac{n^n}{(n+\gamma+\delta)^n}\leq C'$$
		Equivalently,
		$$ \mu_n \leq C' \left(\frac{n+\gamma+\delta}{n}\right)^{n+\gamma+\delta}\frac{1}{n^{1-\delta}}$$
		Since $\lim_n \left(\frac{n+\gamma+\delta}{n}\right)^{n+\gamma+\delta}=e^{\gamma+\delta},$ we get $\mu_n=O(\frac{1}{n^{1-\delta}}).$
		
		\noindent
		(iii)$\Rightarrow$(iv) Let us see that for every $s>0,$ if $\mu_n=O(\frac{1}{n^{s}}),$ then
		$\mu $ is an $s$-Carleson measure. Assume there exists $C>0$ such that $\mu_n\leq Cn^{-s}$ for every $n\in \N.$   It is enough to see that  there exists $D>0$ such that $ \mu([t,1))\leq D(1-t)^{s}$ for every $t\in [\frac{1}{2}, 1).$  Consider  $ \frac{1}{2}\leq t<1$ and take $n\in \N,$ $n\geq 2,$ such that  $1-\frac{1}{n}\leq t<1-\frac{1}{n+1}.$ Thus, there exists $C'>0$ such that $t^{n}\geq \left(1-\frac{1}{n}\right)^{n}\geq C'.$   Hence,  $$\mu([t,1))\leq\mu[1-\frac{1}{n},1)\leq \frac{1}{C'}\int_{1-\frac{1}{n}}^{1}t^{n}d\mu(t)\leq \frac{C}{C'}n^{-s}\leq D(n+1)^{-s}\leq D(1-t)^{s}$$
		for some $D>0,$ as we wanted to see.
		
		\noindent
		(iv)$\Rightarrow$(v) By hypothesis, there exists $C>0$ such that $\mu[t,1)\leq C(1-t)^{1-\delta}.$ Thus, using integration by parts we get
		\begin{eqnarray*}
			\int_0^1\frac{d\mu(t)}{(1-tr)^{\gamma+1}}&=&\mu([0,1))+\int_0^1 \left(\frac{1}{(1-tr)^{\gamma+1}}\right)'\mu([t,1))dt\\
			&= &\mu([0,1))+r(\gamma+1)\int_0^1 \frac{\mu([t,1))}{(1-tr)^{\gamma+2}}dt\\
			&\leq& \mu([0,1))+Cr(\gamma+1)\int_0^1 \frac{(1-t)^{1-\delta}}{(1-tr)^{\gamma+2}}dt\\ 	
			&\leq& \mu([0,1))+Cr(\gamma+1)\int_0^1 \frac{dt}{(1-tr)^{\gamma+\delta+1}}\\
			&\leq& 	\mu([0,1))+C\frac{\gamma+1}{\gamma +\delta}\ \frac{1}{(1-r)^{\gamma +\delta}},
		\end{eqnarray*}
		and the conclusion holds.
		
		\noindent
		(v)$\Rightarrow$(i) is a particular case of Proposition \ref{Caract_cont_gral}.\\
		To prove that the first five equivalent conditions are equivalent to (vi), by  Theorem \ref{lemma_cont_Hv_Hv0} we only need to show  that $C_\mu(1)\in H_{\vgd}^0$ when (i) holds. By the equivalence between conditions (i) and (iii), the condition $C_\mu(1)\in H_{\vgd}^0$ is satisfied for every $\gamma>0$   by Proposition \ref{exemples C_mu(1)}(i) when $-\gamma <\delta\leq 0$ and   by Proposition \ref{exemples C_mu(1)}(ii) when $\delta\in (0,1)$ (just take $s=1-\delta$ and consider $H_{\vgd}^0$ instead of $H_{\vg}^0$ in the proposition).
	\end{proof}

	The following theorem characterize the compactness of the operator and provides also a characterization of vanishing $s$-Carleson measures in terms of integrals, complementing the results in \cite[Proposition 1]{CGP},  \cite[Lemma 2]{GGM} and \cite[Corollary 4.6]{Blasco}.
\begin{theorem}\label{Teorema compact}
	Let $\mu$ be a positive finite Borel measure on $[0,1)$  and $\gamma>0, \delta\in (-\gamma,1)$.  The following conditions are equivalent:
	\begin{itemize}
		\item[(i)] $C_\mu: H_{\vg}^0\rightarrow H_{\vgd}^0$ is a compact operator
		\item[(ii)] $C_\mu: H_{\vg}^\infty\rightarrow H_{\vgd}^\infty$  is a compact operator
		\item[(iii)] $C_\mu(H_{\vg}^\infty)\subseteq  H_{\vgd}^0$
		\item[(iv)] $C_\mu\left(\frac{1}{(1-z)^{\gamma}}\right)\in H_{\vgd}^0$
		\item[(v)] $\mu_n=o(\frac{1}{n^{1-\delta}})$
		\item[(vi)] $\mu $ is a vanishing $(1-\delta)$-Carleson measure
		\item[(vii)] $\lim_{r\rightarrow 1^-}(1-r)^{\gamma+\delta} \int_{0}^{1}\frac{d\mu(t)}{(1-tr)^{\gamma+1}}=0$
	\end{itemize}
\end{theorem}

\begin{proof}
	\noindent
	(i) $\Rightarrow$ (ii) follows from  Theorem \ref{lemma_cont_Hv_Hv0}, since the  bitranspose of $C_\mu: H_{\vg}^0\rightarrow H_{\vgd}^0$ is $C_\mu: H_{\vg}^\infty\rightarrow H_{\vgd}^\infty.$\\
	(ii) $\Rightarrow$ (iii) follows from  Proposition \ref{weaklycompact} and the equivalence between (i) and (vi) in Theorem \ref{Teorema cont}.\\ (iii)$\Rightarrow$(iv)  is trivial, as $ \frac{1}{(1-z)^{\gamma}}\in H_{\vg}^\infty$.\\
	\noindent
	(iv) $\Rightarrow$ (v)
	We fix $\eps>0$. The function $f(z)=\frac{1}{(1-z)^{\gamma}},\ z\in \D,$  satisfies $f(z)=\sum_{n=0}^{\infty}a_n(\gamma)z^n,\ z\in \D,$ with $a_n(\gamma) >0,$ $a_n(\gamma)  \cong n^{\gamma-1}$ (eg.\cite[p.12]{GGM}).  By  the proof of Theorem \ref{Teorema cont} (ii) $\Rightarrow $ (iii),  there exists $c>0$ such that $ \sum_{k=1}^{n}\sum_{j=1}^{k}j^{\gamma-1}\geq c n^{\gamma +1}.$  Let $\eps'>0$ such that
	
	$$\frac{\eps'}{c(\gamma+\delta)^{\gamma+\delta}} \sup_n\left(\frac{n+\gamma+\delta}{n}\right)^{n+\gamma+\delta}<\eps.$$
	
\noindent If  $C_\mu\left(\frac{1}{(1-z)^{\gamma}}\right)\in H_{\vgd}^0,$   by Lemma \ref{PropNeces1}(ii),  there is  $n_0\in \N$ and $r_0\in [0,1)$ such that, for every $n\geq n_0$,
	$$\mu_n\left(\sup_{r_0\leq r<1}  \vgd(r)r^n\right)\left(\sum_{k=0}^{n}\sum_{j=0}^{k}j^{\gamma-1}\right)\leq \eps'.$$
	The supremum $\sup_{0\leq r<1}  \vgd(r)r^n=\|z^n\|_{\vgd}=\frac{(\gamma+\delta)^{\gamma+\delta}}{(n+\gamma+\delta)^{\gamma+\delta}}\frac{n^n}{(n+\gamma+\delta)^n}$ is attained in $s_n=\frac{n}{n+\gamma+\delta}.$ So, as $\lim_ns_n=1,$  there exists $n_1\in \N,$ $n_1\geq n_0,$ such that $s_n\geq r_0$ for every $n\geq n_1.$ Therefore, for  $n\geq n_1$ we get
	
	 $$\mu_n\frac{(\gamma+\delta)^{\gamma+\delta}}{(n+\gamma+\delta)^{\gamma+\delta}}\frac{n^n}{(n+\gamma+\delta)^n}  n^{\gamma +1} \leq \frac{\eps'}{c},$$
	which implies
	$$ n^{1-\delta}\mu_n \leq \frac{\eps'}{c(\gamma+\delta)^{\gamma+\delta}} \left(\frac{n+\gamma+\delta}{n}\right)^{n+\gamma+\delta}\leq \eps,$$
	as we wanted to see.\\
	\noindent
	(v)$\Rightarrow$(vi) Let us see that for every $s>0,$ if $\mu_n=o(\frac{1}{n^{s}}),$ then
	$\mu $ is a vanishing $s$-Carleson measure.  By hypothesis, given $\eps>0$  there exists $n_0\in \N,$ $n_0\geq 2$ such that $\mu_n\leq  e^{-1}\eps (n+1)^{-s}$ for every $n\geq n_0.$ Fix  $t\in [1-{\frac{1}{n_0}},1)$ and take $n\geq n_0$ such that  $1-\frac{1}{n}\leq t<1-\frac{1}{n+1}.$ We get $t^{n}\geq \left(1-\frac{1}{n}\right)^{n}\geq e^{-1},$     and thus,
	 $$\mu([t,1))\leq\mu[1-\frac{1}{n},1)\leq e\int_{1-\frac{1}{n}}^{1}t^{n}d\mu(t)\leq e  \mu_n\leq \eps (n+1)^{-s}\leq \eps(1-t)^{s},$$
	as we wanted to see.
	
	\noindent
	(vi)$\Rightarrow$(vii) Given $\eps>0$, take  $\eps'>0$ such that $\eps'(\gamma+1)/(\gamma+\delta)<\eps/3$ and choose $t_0\in [0,1)$ such that $\mu([t,1))\leq\eps'(1-t)^{1-\delta}$ for every $t\geq t_0.$ As
	
	\begin{eqnarray*}\int_0^1 \frac{1}{(1-tr)^{\gamma+1}}d\mu(t)
		\leq \frac{1}{(1-t_0)^{\gamma+1}}\mu[0,1)+ \int_{t_0}^{1}\frac{1}{(1-tr)^{\gamma+1}}d\mu(t),
	\end{eqnarray*}
	integrating by parts, we get

	\begin{eqnarray*}
		 \int_{t_0}^{1}\frac{d\mu(t)}{(1-tr)^{\gamma+1}}&=&\frac{\mu[t_0,1)}{(1-rt_0)^{\gamma+1}}+r(\gamma+1)\int_{t_0}^1\frac{\mu[t,1)}{(1-rt)^{\gamma+2}}dt\\
		&	\leq& \frac{\mu(t_0,1]}{(1-t_0)^{\gamma+1}}+\eps'r(\gamma+1)\int_{0}^{1}\frac{(1-t)^{1-\delta}}{(1-rt)^{\gamma+2}}dt\\
		&\leq& \frac{\mu(0,1]}{(1-t_0)^{\gamma+1}}+\eps'r(\gamma+1)\int_{0}^{1}\frac{dt}{(1-rt)^{\gamma +\delta+1}}\\
		&\leq &\frac{\mu(0,1]}{(1-t_0)^{\gamma+1}}+\eps'\frac{(\gamma+1)}{\gamma+\delta} \frac{1}{(1-r)^{\gamma+\delta}}\\
		&\leq& \frac{\mu(0,1]}{(1-t_0)^{\gamma+1}}+\frac{\eps}{3} \frac{1}{(1-r)^{\gamma+\delta}}.
	\end{eqnarray*}
	\noindent
	If we choose $0<r_0<1$ such that $(1-r)^{\gamma+\delta}\frac{\mu(0,1]}{(1-t_0)^{\gamma+1}}\leq \frac{\eps}{3}$ for each $r_0\leq r< 1$, (vii) is satisfied.
	
	\noindent
	(vii)$\Rightarrow$(i) follows by Proposition \ref{Caract_cont_gral}(ii).
\end{proof}

The existence of the probability measure satisfying the conditions in Corollary \ref{exemple_p} below is due to Example \ref{exp} in Appendix A.

\begin{corollary}\label{exemple_p}
Fix $\gamma>0$ and $0<p<\gamma+1$. Let $\mu_n=(n+1)^{-p}$ and consider  the probability measure $\mu$  which is the solution of the Hausdorff moment problem for $(\mu_n)_{n\in\N_0}$. Then $C_\mu: H_{v_{\gamma}}\to H_{v_{\gamma+1-p}}$ is continuous, not compact, and
$C_\mu: H_{v_{\gamma}}\to H_{v_{\gamma+1-t}}$ is not continuous for any $p<t<\gamma+1$.
\end{corollary}

\begin{proof}
	It follows immediately by applying Theorems \ref{Teorema cont} and \ref{Teorema compact}  for $\delta=1-p$ and $\delta=1-t$.
\end{proof}

We remark that this corollary, for $p=1$ shows the non compactness of the classical Ces\`aro operator $C$ in any $H_{v_\gamma}$, result which follows from the description of the spectrum given  in \cite{Persson2008}.

\section{The spectrum of $C_\mu$ on spaces of holomorphic functions}\label{sec:Spectrum}

We recall that given a continuous linear operator $T$ on a locally convex space $E,$ the \emph{spectrum} of $T$ is the set $\sigma(T):=\{\lambda\in \C:\ \lambda I-T \text{ is not invertible}\}$. When the space $E$ is Fr\'echet, Banach or regular (LB),   i.e. a countable inductive limit of Banach spaces in which each bounded set is contained and bounded in some step, the closed graph theorem  \cite[Theorem 24.31]{meisevogt} implies that $\sigma(T):=\{\lambda\in \C:\ \lambda I-T \text{ is not bijective}\}$. The \textit{point spectrum} of $T$ is defined as  $\sigma_p(T):=\{\lambda\in \C:\ \lambda I-T \text{ is not injective}\}$. 	 When $E$ is Banach and $T$ is compact, the Riesz theorem ensures that $\sigma_p(T)$ is countable, $\sigma(T)=\overline{\sigma_p(T)}$ and $\overline{\sigma_p(T)}\setminus \sigma_p(T)\subseteq \{0\}$. The set $\rho(T)=\C\setminus \sigma(T)$ is called the \textit{resolvent} of $T.$  If we want to distinguish the space, we write $\sigma(T, E),$  $\sigma_p(T, E)$ and $\rho(T, E).$

\subsection{The spectrum of $C_\mu: H(\D)\to H(\D)$}
\begin{proposition}\label{Espectre HD}
Let $\mu$ be a positive finite Borel measure on $[0,1)$ such that $\mu((0,1))>0$ and let $C_\mu: H(\D)\to H(\D)$   be well defined (and then continuous). The point spectrum satisfies  $\sigma_p(C_\mu, H(\D))=\{\mu_n:\ n\in\N_0\}$ and the eigenspace associated to $\mu_n$ is generated by the function
\begin{equation}\label{eigenfunctions}
f_n(z)=\sum_{k=n}^{\infty} a_kz^k, \  a_k=\frac{\mu_k \mu_{n}^{k-n-1}}{\prod_{j=n+1}^{k}(\mu_{n}-\mu_j)}, \ k\geq n+1, \ a_n=1.
\end{equation}
\end{proposition}

\begin{proof}
	If $\lambda\in \C$ is an eigenvalue, then there exists $f=\sum_{n=0}^{\infty}a_nz^n\in  H(\D),$ $f\neq 0,$ such that
$$\sum_{n=0}^{\infty}\mu_n(\sum_{k=0}^{n}a_k)z^n=C_{\mu}(f)=\lambda f=\sum_{n=0}^{\infty}\lambda a_nz^n.$$
Thus, for every $n\in \N$  we get   $\lambda a_n =\mu_n \sum_{k=0}^{n}a_k,$ or equivalently, $a_0(\lambda-\mu_0)=0,$   $a_n(\lambda - \mu_n)=\mu_n\sum_{k=0}^{n-1}a_k,$ $n\geq 1.$
If $a_n=0$ for every $n\leq n_0-1$ and $a_{n_0}\neq 0,$ then we get  $a_{n_0}(\lambda - \mu_{n_0})=0,$ and thus, $\lambda=\mu_{n_0}.$ This yields $\sigma_p(C_\mu, H(\D))\subseteq \{\mu_n:\ n\in\N_0\}.$ Moreover, as  $(\mu_n)_n$ is strictly decreasing by Lemma \ref{moments}(iii),
\begin{equation}\label{induction3}
	a_n=\frac{\mu_n}{\mu_{n_0} - \mu_n}\sum_{k=0}^{n-1}a_k,\ n\geq n_0+1,
\end{equation}
which implies $\sum_{k=0}^{n}a_k=a_n\frac{\mu_{n_0}-\mu_n}{\mu_n}+a_n=\frac{\mu_{n_0}}{\mu_n}a_n.$
Using this fact and (\ref{induction3}), we get by induction that if $f_{n_0}$ is an eigenfunction associated to $\mu_{n_0},$ it must be of the form
$$f_{n_0}=\sum_{n=n_0}^{\infty} a_nz^n, \text{ where }  a_n=\frac{\mu_n \mu_{n_0}^{n-n_0-1}}{\prod_{j=n_0+1}^{n}(\mu_{n_0}-\mu_j)}a_{n_0}, \ n\geq n_0+1, \ a_{n_0}\in \C.
$$
Now, observe that
$$\frac{a_n}{a_{n+1}}=\frac{\mu_{n}}{\mu_{n+1}}\frac{\mu_{n_0}-\mu_{n+1}}{\mu_{n_0}}>0,$$
hence, by Lemma \ref{moments} (i) and (ii)  the radius of convergence of $f_{n_0}$ is
$$R= \lim_n \frac{a_n}{a_{n+1}}=\lim_n  \frac{\mu_{n}}{\mu_{n+1}}\geq 1,$$
which implies $f_{n_0}\in H(\D)$ for every $a_{n_0}\in \C$. As this is satisfied  for every $n_0\in \N_0,$ we get $\sigma_p(C_\mu, H(\D))=\{\mu_n:\ n\in\N_0\}.$
\end{proof}

	In the case of the Cesàro operator $C: H(\D)\to H(\D),$ we get that the eigenspace associated to the eigenvalue $\mu_n=\frac{1}{n+1}$ is generated by the eigenfunction $\frac{z^{n}}{(1-z)^{n+1}},$ as obtained in  \cite{Persson2008}.

\begin{proposition}\label{C_mu sobre HD}
	Let $\mu$ be a positive finite Borel measure on $[0,1)$ such that $\mu((0,1))>0$ and consider $C_\mu: H(\D)\to H(\D)$  well defined (and then continuous). The following are equivalent:
	\begin{itemize}
		\item[(i)] $0\notin \sigma(C_{\mu}, H(\D)).$
		\item[(ii)] 	$C_\mu: H(\D)\to H(\D)$  is surjective.
		\item[(iii)] $M_{(1/\mu_n)_n}: H(\D)\to H(\D),$ $\sum_{n=0}^{\infty}b_nz^n\mapsto\sum_{n=0}^{\infty}\frac{b_n}{\mu_n}z^n$ is continuous.
		\item[(iv)] For every $r\in (0,1)$ there exist $s\in (r,1)$ and $C>0$ such that $\frac{1}{\mu_n}\leq C \left(\frac{s}{r}\right)^n$ for every $n\in \N.$
		\item[(v)] For every $\eps\in (0,1)$ there exists   $C>0$ such that $\frac{1}{\mu_n}\leq C(1+\eps)^n$ for every $n\in \N.$
	\end{itemize}
\end{proposition}

\begin{proof}
	(i) is  equivalent to (ii) because $0\notin \sigma_p(C_\mu, H(\D))$ by Proposition \ref{Espectre HD}.\\
	(ii) $\Leftrightarrow$ (iii)  Given $f(z)=\sum_{n=0}^{\infty}b_nz^n\in H(\D),$ consider the problem of finding $g_f(z)=\sum_{n=0}^{\infty}a_nz^n\in H(\D)$ such that
	 $$\sum_{n=0}^{\infty}\mu_n(\sum_{k=0}^{n}a_k)z^n=C_{\mu}(g_f)(z)=f(z)=\sum_{n=0}^{\infty}b_nz^n. $$
This yields $\sum_{k=0}^{n}a_k=\frac{b_n}{\mu_n}$ for every $n\in \N_0.$ As we can assume without loss of generality that $\mu_0=1,$  we obtain
	\begin{equation*}
		 a_n=\sum_{k=0}^{n}a_k-\sum_{k=0}^{n-1}a_k=\frac{b_n}{\mu_n}-\frac{b_{n-1}}{\mu_{n-1}},\  n\in \N,\ a_0=b_0.
	\end{equation*}
	As a consequence, 	$C_\mu: H(\D)\to H(\D)$  is surjective if and only if
	\begin{equation}\label{antiimatgeHD}
		g_f(z)=b_0+\sum_{n=1}^{\infty}\left( \frac{b_n}{\mu_n}-\frac{b_{n-1}}{\mu_{n-1}} \right)z^n\in H(\D)
	\end{equation}
	for every $f(z)=\sum_{n=0}^{\infty}b_nz^n\in H(\D).$ And this condition is equivalent to  saying that the operator
	\begin{equation}\label{multipliers surj HD}
		M_{(1/\mu_n)_n}: H(\D)\to H(\D),\ f(z)=\sum_{n=0}^{\infty}b_nz^n\mapsto h_f(z)=\sum_{n=0}^{\infty}\frac{b_n}{\mu_n}z^n
	\end{equation}
	is continuous. Indeed,    $g_f(z)=(1-z)h_f(z)$ and then $h_f(z)=\frac{1}{1-z}g_f(z)$. Hence $g_f$ is well defined and holomorphic on $\D$ if and only if $h_f$ is so.\\
	(iii) $\Leftrightarrow$ (iv)   First, observe that $M_{(1/\mu_n)_n}: H(\D)\to H(\D)$ is continuous if and only if for every $r\in (0,1)$ there exist $s\in (r,1)$  and $C>0$ such that $\sup_n\frac{|b_n|}{\mu_n}r^n\leq C\sup_n|b_n|s^n$ for every  $\sum_{n=0}^{\infty}b_nz^n\in H(\D).$ And this condition is equivalent to (iv) (by evaluation on the monomials). \\
	(iv) $\Rightarrow$ (v)   follows easily if we take $r=\frac{1}{1+\eps}.$ For the converse, given $0<r<1,$ consider $0<\eps< \frac{1-r}{r},$ $\epsilon<1,$  and  $s\in [r(1+\eps),1).$ Now, the hypothesis yields     $C>0$ such that $\frac{1}{\mu_n}\leq C(1+\eps)^n\leq C\left(\frac{s}{r}\right)^n$ for every $n\in \N.$
\end{proof}

\begin{proposition}\label{C_mu-lambdaI sobre HD}
	Let $\mu$ be a positive finite Borel measure on $[0,1)$ such that $\mu((0,1))>0$ and consider $C_\mu: H(\D)\to H(\D).$
 \begin{itemize}
     \item [(a)]For $\lambda\in \{\mu_n: \ n\in\N_0 \},$ the operator $C_\mu-\lambda I: H(\D)\to H(\D)$  is not injective and not surjective. In fact, $z^k\notin(C_\mu-\mu_k I)(H(\D)),$ $k\in \N_0.$
     \item [(b)] For $\lambda\notin \{\mu_n: \ n\in\N_0 \}\cup \{0\},$ the following are equivalent:
	\begin{itemize}
		\item[(i)] $\lambda\notin \sigma(C_{\mu}, H(\D)).$
		\item[(ii)] 	$C_\mu-\lambda I: H(\D)\to H(\D)$  is  surjective.
		\item[(iii)]  For every $\sum_{k=0}^{\infty}b_kz^k\in H(\D),$
		\begin{equation}\label{eq sobre 1}
			 \sum_{k=0}^{\infty}\mu_k\left(\sum_{n=0}^{k-1}\frac{b_n}{\prod_{j=n}^{k}\left(1-\frac{\mu_j}{\lambda}\right)}\right)z^k\in H(\D)
		\end{equation}
	\end{itemize}
\end{itemize}
\end{proposition}

\begin{proof}
(a) First, take $k\in \N_0$ and  consider $C_\mu-\mu_k I: H(\D)\to H(\D).$ By Proposition \ref{Espectre HD} we know that $C_\mu-\mu_kI$  is not injective. Let us see that it is not surjective either by showing that  $z^k$ does not belong to the image of $ C_\mu-\mu_k I.$ Indeed, if $f(z)=\sum_{n=0}^{\infty}a_nz^n\in H(\D)$ is such that  $(C_\mu-\mu_kI)(f)(z)=z^k,$ then
\begin{eqnarray*}
z^k=\sum_{n=0}^{\infty}\left(\mu_n(\sum_{j=0}^{n}a_j)-\mu_ka_n\right)z^n.
\end{eqnarray*}

Since $(\mu_n)_{n\in\N_0}$ is a sequence of strictly positive numbers, there is no $(a_n)_{n\in\N_0}\subseteq \C$ such that  $\mu_n(\sum_{j=0}^{n}a_j)-\mu_ka_n=0$ for $n\neq k$ and $\mu_k(\sum_{j=0}^{k}a_j)-\mu_ka_k=1$. Hence $C_\mu-\mu_k I$ is not surjective.

(b)  We take $\lambda\notin \{\mu_n: \ n\in\N_0 \}\cup \{0\},$
$f(z)=\sum_{k=0}^{\infty}b_kz^k\in H(\D),$ and consider the problem of finding $g_f(z)=\sum_{k=0}^{\infty}a_kz^k$ such that
	 $$\sum_{k=0}^{\infty}\mu_k(\sum_{n=0}^{k}a_n)z^k-\lambda\sum_{k=0}^{\infty}a_kz^k =(C_{\mu}-\lambda I)(g_f)(z)=f(z)=\sum_{k=0}^{\infty}b_kz^k. $$
	This implies $\mu_k\sum_{n=0}^{k}a_n-\lambda a_k=b_k$ for every $k\in \N_0,$ that is, $a_k(\mu_k-\lambda)=b_k-\mu_k\sum_{n=0}^{k-1}a_n.$ Assuming without loss of generality that $\mu_0=1$ we get
	 $$a_k=\frac{b_k}{\mu_k-\lambda}-\frac{\mu_k}{\mu_k-\lambda}\sum_{n=0}^{k-1}a_n ,\ k\in \N, \  a_0=\frac{b_0}{\mu_0-\lambda}.$$
 	Inductively, for every $k\in \N,$ we get
	\begin{eqnarray}
		 a_k&=&\frac{b_k}{\mu_k-\lambda}+\mu_k\sum_{n=0}^{k-1}(-1)^{k-n}\frac{\lambda^{k-n-1}b_n}{\prod_{j=n}^{k}(\mu_j-\lambda)}\nonumber\\ &=&\frac{b_k}{\mu_k-\lambda}-\frac{\mu_k}{\lambda^2}\sum_{n=0}^{k-1}\frac{b_n}{\prod_{j=n}^{k}(1-\frac{\mu_j}{\lambda})}.
	\end{eqnarray}
	Now, observe that $\sum_{k=0}^{\infty}\frac{b_k}{\mu_k-\lambda}z^k\in H(\D)$ since $|\mu_k-\lambda|>\delta$ for every $k\in \N$  for some $\delta>0$.  Therefore, $g_f(z)\in H(\D)$ if and only if $\sum_{k=0}^{\infty}\mu_k\left(\sum_{n=0}^{k-1}\frac{b_n}{\prod_{j=n}^{k}\left(1-\frac{\mu_j}{\lambda}\right)}\right)z^k\in H(\D)$.
This gives the equivalence (ii)$\Leftrightarrow$(iii). Observe that (i)$\Leftrightarrow$(ii) is trivial, since $\lambda\notin \sigma_p(C_\mu, H(\D)).$
\end{proof}

\begin{remark}
 Condition (\ref{eq sobre 1}) is satisfied, for instance, if $\sum_{n=0}^{\infty}\mu_n<\infty.$ Indeed, in this case the infinite product converges, and thus the radius of convergence of the power series is  that of $\sum_{k=0}^{\infty}\mu_k(\sum_{n=0}^{k}|b_n|)z^k$, which belongs to $H(\D)$ because $f(z)=\sum_{k=0}^{\infty}|b_k|z^k\in H(\D)$ and $C_{\mu}$ is  well defined, and then  continuous. In the next proposition we prove that this assumption can be relaxed.
\end{remark}

\begin{lemma}\label{Carleson_estim_inf_prod}
Let $\mu$ be a positive finite Borel measure on $[0,1).$  The following assertions are satisfied:
\begin{itemize}
    \item [(i)]  Fix $n\in \N.$ If  $C>0$  satisfies that there exists $j_0\in \N, j_0\geq  n$ such that  $\mu_j\leq \frac{C}{j}$ for every $j> j_0,$ then for  every $k\in \N,$ $k\geq n$    and $\lambda\in \C,  $ $\lambda  \notin \{\mu_j: \ j\geq n \}\cup \{0\},$ we get
\begin{eqnarray}
		\prod_{j=n}^{k}\left|1-\frac{\mu_j}{\lambda}\right|\gtrsim \frac{1}{k^{\lceil C|Re(1/\lambda)|\rceil}},
	\end{eqnarray}
where  $\lceil x\rceil$ denotes the least integer greater than or equal to  $x>0$. Here, the constant giving the lower estimate can depend on $C$ and on $n.$ \item [(ii)] Fix $n\in \N.$  If $D>0$ satisfies that there exists  $j_0\in\N$  $j_0\geq n,$ satisfying $\mu_j\geq \frac{D}{j}$  for every $j> j_0,$    then for  every $k\in \N,$ $k\geq n$  and    $a>\mu_n,$ we get
$$\prod_{j=n}^{k}\left(1-\frac{\mu_j}{a}\right)\lesssim \left(\frac{1}{k}\right)^{\lfloor D/a\rfloor},$$
where $\lfloor x\rfloor$ denotes the biggest integer smaller than or equal to  $x>0$. Here, the constant giving the upper estimate can depend on $D$ and on $n.$
 \end{itemize}
\end{lemma}

\begin{proof}
(i) Fix $n\in \N$  and consider  $\lambda,$  $C$ and $j_0\in \N,$ $j_0\geq  \max\{n, \lceil C|Re(1/\lambda)|\rceil\},$ as in the hypothesis.  Then, for each $k\geq  j_0+1,$ we get
	\begin{eqnarray}
		\prod_{j=n}^{k}\left|1-\frac{\mu_j}{\lambda}\right|&\gtrsim&
	\prod_{j=j_0+1}^{k}\left|1-\frac{\mu_j}{\lambda}\right|\geq
	\prod_{j=j_0+1}^{k}\left|Re\left(1-\frac{\mu_j}{\lambda}\right)\right| \nonumber\\  &\geq&\prod_{j=j_0+1}^{k}\left(1-\frac{C}{j}\left|Re\left(\frac{1}{\lambda}\right)\right|\right)\geq
		\prod_{j=j_0+1}^{k}\frac{\left(j- \lceil C|Re(1/\lambda)|\rceil)\right)}{j}\nonumber\\
		&\geq  & \frac{(k-\lceil C\left|Re\left(1/\lambda\right)\right|\rceil)!}{k!}  \geq\frac{1}{k^{\lceil C\left|Re\left(1/\lambda\right)\right|\rceil}}.\nonumber
	\end{eqnarray}
 (ii) Fix $n\in \N$   and consider $a$, $j_0$ and    $D$ as in the hypothesis (consider the minimun  $j_0\geq n$ satisfying the condition). For $k\geq j_0+1$, we get

 	\begin{eqnarray*}
 	\prod_{j=n}^{k}\left(1-\frac{\mu_j}{a}\right)  &\lesssim& \prod_{j=j_0+1}^{k}\left(1-\frac{\lfloor D/a\rfloor}{j}\right)=\prod_{j=j_0+1}^{k}\frac{j-\lfloor D/a\rfloor}{j}=\frac{j_0!}{(j_0-\lfloor D/a\rfloor)!}
 	\frac{(k-\lfloor D/a\rfloor)!}{k!} \\
 	&\leq&\frac{{j_0}^{\lfloor D/a\rfloor}}{(k-\lfloor D/a|\rfloor)^{\lfloor D/a\rfloor}}\lesssim \left(\frac{1}{k}\right)^{\lfloor D/a\rfloor},
 \end{eqnarray*}
where the constant giving the upper estimate can depend on $D$ and $n.$
\end{proof}

\begin{remark} \label{remarkl}
In Lemma \ref{Carleson_estim_inf_prod}(i), if we assume that $\mu_j\leq \frac{C}{j}$ for all $j\in\N$, we get  $\prod_{j=n}^{k}\left|1-\frac{\mu_j}{\lambda}\right|\geq M \frac{1}{k^{\lceil C|Re(1/\lambda)|\rceil}}$, where $M>0$ does not depend  on $k,$ nor on $n$.

\end{remark}
\begin{proposition}\label{Espectre C_mu HD}
	Let $\mu$ be a positive finite Borel measure on $[0,1)$ such that $\mu((0,1))>0$ and consider $C_\mu: H(\D)\to H(\D).$ If $\mu$ is a Carleson measure, then  $\sigma(C_{\mu}, H(\D))\setminus\{0\}=\{\mu_n: \ n\in\N_0 \}.$
\end{proposition}

\begin{proof}
	By Proposition \ref{Espectre HD} we get $ \{\mu_n: \ n\in\N_0 \}\subseteq \sigma(C_{\mu}).$ Let us see $\sigma(C_{\mu})\setminus\{0\}\subseteq\{\mu_n: \ n\in\N_0 \}.$  If $\mu$ is a Carleson measure, we can find $C>0$ such that $\mu_j\leq \frac{C}{j}$ for every $j\in \N.$ Let $\lambda\notin \{\mu_n: \ n\in\N_0 \}\cup \{0\}.$
 By Lemma \ref{Carleson_estim_inf_prod}(i) and Remark \ref{remarkl}, we get $M>0$ such that, for $n\in\N$ and $k\geq n,$
 \begin{eqnarray}
		\prod_{j=n}^{k}\left|1-\frac{\mu_j}{\lambda}\right|\geq \frac{M}{k^{\lceil C|Re(1/\lambda)|\rceil}}.
	\end{eqnarray}
	Therefore, for every  $f(z)=\sum_{k=0}^{\infty}b_kz^k\in H(\D)$  and every $k\in \N$
	 $$\mu_k\left|\sum_{n=0}^{k-1}\frac{b_n}{\prod_{j=n}^{k}\left(1-\frac{\mu_j}{\lambda}\right)}\right|\leq \frac{1}{M} k^{\lceil C|Re(1/\lambda)|\rceil}\mu_k\sum_{n=0}^{k-1}|b_n|.$$
	Since   $\tilde{f}(z)=\sum_{k=0}^{\infty}|b_k|z^k\in H(\D)$, and $C_{\mu}(\tilde{f})=\sum_{k=0}^{\infty}\mu_k\left(\sum_{n=0}^{k}|b_n|\right)z^k\in H(\D),$  the radius of convergence formula yields
	\begin{eqnarray*}
		 \sum_{k=0}^{\infty}\mu_k\left(\sum_{n=0}^{k-1}\frac{b_n}{\prod_{j=n}^{k}\left(1-\frac{\mu_j}{\lambda}\right)}\right)z^k\in H(\D).
	\end{eqnarray*}
	We conclude by Proposition \ref{C_mu-lambdaI sobre HD}.
	\end{proof}

\begin{example}
	\begin{itemize}
		\item[(i)]	 If $\mu$ is a positive finite Borel measure on $[0,1)$ such that $\mu_n\cong \frac{1}{n^s},$ $s\geq 1,$ then $\sigma(C_{\mu}, H(\D))=\{\mu_n: \ n\in\N_0 \}.$  In particular, this is satisfied for the Cesàro operator $C.$
		
		  \item [(ii)]  The positive  measure $\mu$ sastisfying $\mu_n=\frac{1}{2^n},$ $n\in \N,$ is  $s$-Carleson for every $s>0,$ satisfies $\sum_n\mu_n<\infty$ and $\sigma(C_{\mu},H(\D))=\{\mu_n: \ n\in\N_0 \}\cup\{0\}.$

\item[(iii)] If $d\mu(t)=\exp(-\frac{\alpha}{(1-t)^{\beta}})dt,$ $\alpha, \beta>0,$ then $\sigma(C_{\mu}, H(\D))=\{\mu_n: \ n\in\N_0 \}.$ Indeed, $\mu_n  \sim n^{-\frac{\beta+2}{2(\beta+1)}}\exp\left(-Bn^{\frac{\beta}{\beta+1}}\right),$ with $B=\alpha^{\frac{1}{\beta+1}}\left(\beta^{\frac{1}{\beta+1}}+\beta^{-\frac{1}{\beta+1}}\right)$, by \cite[Lemma 4.28]{Arroussi} and \cite[Lemma 1]{Dostanic}.
		 	\end{itemize}
\end{example}


\subsection{The spectrum of $C_\mu: H_{v_\gamma}^{\infty}\to H_{v_\gamma}^{\infty}$}
\begin{proposition}\label{mu_n_dins_EspectreHv}
Let $X\subseteq H(\D)$ be a Banach space of analytic functions and $\mu$  a positive finite Borel measure on $[0,1)$ such that $\mu((0,1))>0$. If   $C_\mu$ is bounded on $X$,  then the following hold:
\begin{itemize}
	\item[(i)] $\sigma_p(C_{\mu},X)\subseteq \{\mu_n: \ n\in\N_0 \}.$
	\item[(ii)]  If $X$ contains the polynomials, then  $\{\mu_n: \ n\in\N_0 \}\cup \{0\}\subseteq \sigma (C_{\mu},X).$
	\item[(iii)] If $X$ contains the polynomials and $C_{\mu}$ is compact on $X,$ then $\sigma_p(C_{\mu},X)=\{\mu_n: \ n\in\N_0 \}$ and $\sigma(C_{\mu},X)=\{\mu_n: \ n\in\N_0 \}\cup \{0\}.$
\end{itemize}
\end{proposition}

\begin{proof}
(i) follows by Proposition \ref{Espectre HD}.
(ii) is a consequence of Proposition \ref{C_mu-lambdaI sobre HD}(a)  (see \cite{Dahlner} and \cite{Persson2008} for the result on the classical Cesàro operator). By the compactness of $C_{\mu},$   (iii) is a direct consequence of (i) and (ii).
\end{proof}

Theorem  \ref{sum_mu_n_implica_compact} and Proposition \ref{mu_n_dins_EspectreHv}  yield the following result for $H^{\infty}(\D)$ and for general weighted Banach spaces of holomorphic functions.

\begin{proposition}\label{Espectre Hinfty}
Let $\mu$ be a positive finite Borel measure on $[0,1)$ such that $\mu((0,1))>0$  and $\sum_{n=0}^{\infty}\mu_n<\infty.$ The operators $C_\mu: H^{\infty}(\D)\to H^{\infty}(\D),$ $C_\mu: H^{\infty}_v\to H^{\infty}_v$ and $C_\mu: H^{0}_v\to H^{0}_v$  satisfy $\sigma_p(C_\mu)=\{\mu_n: \ n\in\N_0\}$  and $\sigma(C_\mu)=\{\mu_n: \ n\in\N_0\}\cup \{0\}.$
\end{proposition}

Analogously, by Theorem \ref{Teorema compact} and Proposition \ref{mu_n_dins_EspectreHv} we get  the following corollary:

\begin{proposition}\label{EspectrevanishingHv}
	Let $\mu$ be a positive finite Borel measure on $[0,1)$ such that $\mu((0,1))>0$ and $\mu$ is vanishing Carleson. Then:
	\begin{itemize}
		\item[(i)] $\sigma_p(C_\mu, H_{\vg})= \sigma_p(C_\mu, H_{\vg}^{0})=\{\mu_n:\ n\in\N_0\},$
		\item[(ii)]$\sigma(C_\mu, H_{\vg})= \sigma(C_\mu, H_{\vg}^{0})=\{\mu_n:\ n\in\N_0\}\cup\{0\}$.
	\end{itemize}
	In particular, this is satisfied if $\mu$ is an  $s$-Carleson measure, for $s>1$.
\end{proposition}


For non necessarily vanishing Carleson measures, we get inclusions for the point spectrum in the spirit of that obtained in \cite{Persson2008} and \cite{AlemanPersson2010} for the classical Ces\`aro operator.

\begin{proposition}\label{Espectre puntual Hv}
Let $\mu$ be a positive finite Borel measure on $[0,1)$ such that $\mu((0,1))>0$  and let $C>0$ such that $\mu_n\leq \frac{C}{n}$ eventually. For $\gamma>0$ we have:

\begin{itemize}
\item[(i)] $\{\mu_n: \lceil\frac{C}{\mu_n}\rceil\leq \gamma \}=\{\mu_n:  \frac{1}{(1-z)^{\lceil\frac{C}{\mu_n}\rceil}}\in  H_{\vg}^\infty\}\subseteq \sigma_p(C_\mu, H_{\vg}^\infty)\subseteq \{\mu_n: n\in \N_0\}$

    \item [(ii)] $\{\mu_n: \lceil\frac{C}{\mu_n}\rceil< \gamma \}=\{\mu_n:  \frac{1}{(1-z)^{\lceil\frac{C}{\mu_n}\rceil}}\in  H_{\vg}^0\}\subseteq \sigma_p(C_\mu, H_{\vg}^0)\subseteq \{\mu_n: n\in \N_0\}$

\end{itemize}
Moreover, if there exists  $D>0$ (consider the maximum) such that $\mu_n\geq \frac{D}{n}$ eventually, then  we even get

\begin{itemize}
    \item [(iii)] $\sigma_p(C_\mu, H_{\vg}^\infty)\subseteq\{\mu_n: \lfloor\frac{D}{\mu_n}\rfloor\leq \gamma \}=\{\mu_n:  \frac{1}{(1-z)^{\lfloor\frac{D}{\mu_n}\rfloor}}\in  H_{\vg}^\infty\}$

    \item [(iv)] $\sigma_p(C_\mu, H_{\vg}^0)\subseteq\{\mu_n: \lfloor\frac{D}{\mu_n}\rfloor< \gamma \}=\{\mu_n:  \frac{1}{(1-z)^{\lfloor\frac{D}{\mu_n}\rfloor}}\in  H_{\vg}^0\}$

\end{itemize}
Here, $\lceil x\rceil$ $(\text{resp.}\lfloor x\rfloor)$ denotes the least (greatest) integer greater (smaller) than or equal to  $x.$
\end{proposition}

\begin{proof}
As   $H_{\vg}^0(\D)\subseteq H_{\vg}^{\infty}(\D)\subseteq H(\D)$, Proposition \ref{mu_n_dins_EspectreHv}(i) yields $\sigma_p(C_{\mu},H_{\vg}^0)\subseteq\sigma_p(C_{\mu},H_{\vg}^\infty)\subseteq\{\mu_n: \ n\in\N_0\},$  and if $\mu_n$ is an eigenvalue,   by Proposition \ref{Espectre HD} the corresponding  eigenspace is generated by the function $f_n(z)=\sum_{k=n}^{\infty}a_kz^k,$ with $a_n=1$ and
$$  a_k=\frac{\mu_k \mu_{n}^{-1}}{\prod_{j=n+1}^{k}(1-\frac{\mu_j}{\mu_n})}, \ k\geq n+1.$$
By Lemma \ref{Carleson_estim_inf_prod}(i)  we get
$\prod_{j=n+1}^{k}\left(1-\frac{\mu_j}{\mu_n}\right)\gtrsim \frac{1}{k^{\lceil\frac{C}{\mu_n}\rceil}}.$ Thus, as $a_k \lesssim k^{\lceil\frac{C}{\mu_n}\rceil-1}$ for every $k\in \N_0,$   and $g_\alpha(z)=\frac{1}{(1-z)^{\alpha}},\ z\in \D,$  satisfies $g_\alpha(z)=\sum_{k=0}^{\infty}c_k(\alpha)z^k,\ z\in \D,$ with  $c_k(\alpha)  \cong k^{\alpha-1}$, we have
$$\sup_{|z|=r}|f_n(z)|\vg(z)\lesssim \sum_{k=n}^{\infty}k^{\lceil\frac{C}{\mu_n}\rceil-1}r^k\vg(r)\lesssim \sup_{|z|=r}\left|\frac{1}{(1-z)^{\lceil\frac{C}{\mu_n}\rceil}}\right|\vg(z).$$
This yields (i) and (ii).  Analogously,  for $n$ fixed and $D$ as in the hypothesis, Lemma \ref{Carleson_estim_inf_prod}(ii) yields
$\prod_{j=n+1}^{k}\left(1-\frac{\mu_j}{\mu_n}\right)\lesssim \left(\frac{1}{k}\right)^{\lfloor D/\mu_n\rfloor}.$
Now, as  $c_k(\lfloor\frac{D}{\mu_n}\rfloor) \lesssim k^{\lfloor\frac{D}{\mu_n}\rfloor-1}\lesssim a_k$ for every $k\in \N_0,$ we have
$$\sup_{|z|=r}|f_n(z)|\vg(z)\gtrsim    \sum_{k=n}^{\infty}k^{\lfloor\frac{D}{\mu_n}\rfloor-1}r^k\vg(r)\gtrsim \sup_{|z|=r}\left|\frac{1}{(1-z)^{\lfloor\frac{D}{\mu_n}\rfloor}}\right|\vg(z)$$
which concludes  (iii) and (iv).
\end{proof}

As an immediate consequence of Proposition \ref{Espectre puntual Hv}, we get the following:

\begin{corollary}\label{fins_n dins espectre_punt_Hv}
	Let $\mu$ be a Carleson measure. For every  $n_0\in \N_0$  there exists $\gamma>0$ such that $ \{\mu_n:  \  n\leq n_0, \ n\in \N_0\}\subseteq\sigma_p(C_\mu, H_{\vg}^{0})\subseteq\sigma_p(C_\mu, H_{\vg}^{\infty})$.
\end{corollary}

In the next example  we show that multiples of the Cesàro operator are not the only  cases satisfying the two conditions in Proposition \ref{Espectre puntual Hv}.

%

\begin{example}
	Let $\mu$ be a Borel probability on $[0,1)$ satisfying  $\mu_n=\frac{n+2}{2(n+1)^2},$ $n\in \N_0$ (see Example \ref{Ex:nomultipleCesaro} in Appendix A). Then $\mu_n\leq \frac{1/2}{n}$ for every $n\in \N$ and $\mu_n\geq \frac{t/2}{n}$ eventually for every $t<1.$ Moreover,  $\lim_{t\to 1}\left\lfloor \frac{t(n+1)^2}{n+2}\right\rfloor=\left\lfloor \frac{(n+1)^2}{n+2}\right\rfloor$ for every $n\in \N_0.$ Fix $\gamma>0$ and let $n_0=\max\{n\in\N: \ \left\lceil \frac{(n+1)^2}{n+2}\right\rceil\leq \gamma\}$. Then Proposition  \ref{Espectre puntual Hv} gives $\{\mu_n:\ n\leq n_0\}\subseteq \sigma_p(C_\mu,H_{\vg}^\infty)\subseteq \{\mu_n:\ n\leq n_0+1\}$.
\end{example}

The example above shows that when $\mu$ is a Borel probability on $[0,1)$ with $\lim n\mu_n=a>0$, then the point spectrum of $C_\mu$ on $H_{v_\gamma}^{\infty}$ is ``similar''  to that of the classical Ces\`aro operator  (see  \cite[Theorem 4.1]{Persson2008} and \cite[Theorems 4.1 and 5.1, Corollaries 2.1 and 5.1]{AlemanPersson2010}).

\begin{remark}
Let $\mu$ be a positive finite Borel measure on $[0,1)$ such that $\mu((0,1))>0.$ If $\mu$ is vanishing Carleson, as $C_{\mu}$ is compact   we know by Proposition \ref{EspectrevanishingHv} that the point spectrum of $C_{\mu}$ on the weighted Banach spaces of holomorphic functions for standard weights consists of all the moments of the measure. Moreover, we have  the following:
\begin{itemize}
	\item[(a)] For $\gamma\geq 1,$ Proposition  \ref{Espectre puntual Hv} gives a direct proof of Proposition \ref{EspectrevanishingHv}. Indeed, for each $n\in\N_0$, let $\eps>0$ such that $\lceil \eps/\mu_n\rceil =1$. Since $\mu$ is vanishing Carleson, we can apply Proposition \ref{Espectre puntual Hv}(i) for $C=\eps$ and we get $\{\mu_n:\ n\in\N_0\}=\sigma_p(C_\mu, H_{v_\gamma})$.
	Now, the compactness given by Theorem \ref{Teorema compact} yields the result on the spectrum.
	\item[(b)] For $\gamma<1,$  the set $\{\mu_n: \lceil\frac{C}{\mu_n}\rceil\leq \gamma \}$ is empty for each $C>0$, but in any case $\sigma_p(C_\mu,H_{\vg})=\{\mu_n:\ n\in\N_0\}$.
	\end{itemize}
\end{remark}

\section{The operator $C_{\mu}$ on Korenblum Fr\'echet and (LB) spaces}\label{F-LB_Korenblum}

In this section we study the Cesàro-type operator $C_\mu$  when it acts  on Korenblum Fr\'echet and (LB) spaces.  We refer to \cite[Chapter 25]{meisevogt} for all the definitions related to the locally convex spaces considered in this section.

For $0<\alpha<\beta,$ then $H_{v_{\alpha}}^{\infty}\subseteq H_{v_{\beta}}^{\infty}$  with continuous inclusions. So, given  $\gamma\geq  0$  we can consider the projective limit
	$$A_+^{-\gamma}:=\cap_{\mu>\gamma}H_{v_{\mu}}^{\infty}=\{f\in H(\D): \ \sup_{z\in \D}(1-|z|)^{\gamma+\frac{1}{k}}|f(z)|<\infty \text{ for every } k\in \N \},$$
	which is a Fréchet Schwartz space under the norms $\|\ \|_{v_{\gamma}+\frac{1}{k}},$ $k\in \N.$
	Analogously, given $\gamma> 0,$ and $ k\geq k_0$ with $ \gamma-\frac{1}{k_0}>0,$ we can consider the inductive limit
	$$A_{-}^{-\gamma}:=\cup_{\mu<\gamma}H_{v_{\mu}}=\{f\in H(\D): \ \sup_{z\in \D}(1-|z|)^{\gamma-\frac{1}{k}}|f(z)|<\infty \text{ for some } k\in \N \},$$
endowed with the finest locally convex topology such that all inclusions $H_{v_{\gamma-\frac{1}{k}}}^{\infty}\subseteq A_{-}^{-\gamma}$ are continuous. With this topology, $A_{-}^{-\gamma}$ is a complete regular (DFS)-space. If $\gamma=\infty,$ then $A_{-}^{-\gamma}$ is the Korenblum space  \cite{Korenblum}  denoted simply by
	$$A^{-\infty}:=\ind_{n\in \N}H_{v_{n}}^{\infty}=\{f\in H(\D): \ \sup_{z\in \D}(1-|z|)^{n}|f(z)|<\infty \text{ for some } n\in \N \}.$$
	Albanese et al. \cite{ABR2018} determined the spectrum of the Cesàro operator $C$ acting on the Korenblum type spaces $A_+^{-\gamma},$  $A_{-}^{-\gamma}$  and $A^{-\infty}$ and used it to prove the operator  	is never compact on  them.

	\begin{proposition}
		Let $\mu$ be a positive finite Borel measure on $[0,1).$ The operator $C_{\mu}: A^{-\infty}\rightarrow A^{-\infty}$ is continuous.
	\end{proposition}
	
	\begin{proof}
		For the weights $v_n(z)=(1-|z|)^n,$ $z\in \D,$  $n\in \N,$ the operator $C_{\mu}: H_{v_n}\rightarrow H_{v_{n+1}}$ is continuous for every $n\in \N$ by Corollary \ref{Corollary_cont_gral}. So, the continuity follows from the continuity of operators on  inductive limits.
	\end{proof}

	\begin{proposition}
		Assume $\mu$ is a positive finite Borel measure on $[0,1).$  The following are equivalent:
		\begin{itemize}
			\item[(i)] 	  $\mu$  is an $\eps$-Carleson measure for every $ 0<\eps<1,$
			\item[(ii)] $C_{\mu}: A_+^{-\gamma}\rightarrow A_+^{-\gamma}$ is continuous for every $\gamma\geq 0,$
			\item[(iii)] $C_{\mu}: A_{-}^{-\gamma}\rightarrow A_{-}^{-\gamma}$ is continuous for every $\gamma> 0.$
			
		\end{itemize}
		
	\end{proposition}
	
	\begin{proof}
		(i) $\Leftrightarrow$ (ii)	 It is known that $C_{\mu}$ is continuous on $A_+^{-\gamma}$ if and only if for every $k\in \N,$ there exists $l\in \N,$ $l\geq k$ such that  $C_{\mu}: H_{v_{\gamma +\frac{1}{l}}}^{\infty}\rightarrow H_{v_{\gamma +\frac{1}{k}}}^{\infty}$ is continuous (see  e.g. \cite[Lemma 2.5]{ABR2019}). So, a necessary condition for the continuity is that $\mu$ is a $(1-(\frac{1}{k}-\frac{1}{l}))$-Carleson measure. As this must be satisfied for every $k\in \N,$ we get $\mu$  must be an $\eps$-Carleson measure for every $ 0<\eps<1.$ Let us see the suficiency.
		Given $k\in \N,$ select $\delta>0 $ with $\delta<\frac{1}{k}.$ For $\frac{1}{l}<\frac{1}{k}-\delta,$ as $\mu $ is a $(1-\delta)$-Carleson measure, Theorem \ref{Teorema cont} yields
		$$C_{\mu}: H_{v_{\gamma +\frac{1}{l}}}^{\infty}\rightarrow H_{v_{\gamma +\frac{1}{l}+\delta}}^{\infty}$$ is continuous.
		Since  $ \gamma +\frac{1}{l}+\delta<\gamma +\frac{1}{k},$ the  continuity follows.
		
		\noindent
		(i) $\Leftrightarrow$ (iii) By the Grothendieck factorization theorem  \cite[Theorem 24.33]{meisevogt}, $C_{\mu}$ is continuous on $A_{-}^{-\gamma}$ if and only if for every $k\in \N,$ there exists $l\in \N,$ $l\geq k$ such that  $C_{\mu}: H_{v_{\gamma -\frac{1}{k}}}^{\infty}\rightarrow H_{v_{\gamma -\frac{1}{l}}}^{\infty}$ is continuous.  Therefore, $\mu$ must be a $(1-(\frac{1}{k}-\frac{1}{l}))$-Carleson measure for every $k\in \N,$ i.e. (i) must be satisfied. Assume now (i) holds. Given $k\in \N,$ select $\delta>0 $ with $\delta<\frac{1}{k}$ and $l\in \N$ such that $\frac{1}{l}<\frac{1}{k}-\delta.$ As $\mu $ is a $(1-\delta)$-Carleson measure, Theorem \ref{Teorema cont} yields
		$$C_{\mu}: H_{v_{\gamma -\frac{1}{k}}}^{\infty}\rightarrow H_{v_{\gamma -\frac{1}{k}+\delta}}^{\infty}$$ is continuous.
		Since  $ \gamma -\frac{1}{k}+\delta<\gamma -\frac{1}{l},$ we get the continuity.		
	\end{proof}

Now we investigate  the spectrum of the Cesàro-type operator  acting on the Korenblum space $A^{-\infty}.$ By Proposition \ref{Espectre HD} and Corollary \ref{fins_n dins espectre_punt_Hv} we get:

\begin{proposition}\label{Espectre Ainfty}
  Let $\mu$ be a positive finite Borel measure on $[0,1)$ such that $\mu((0,1))>0$ and consider $C_\mu: A^{-\infty}\to A^{-\infty}.$ If $\mu$ is a Carleson measure, then, $\sigma_p(C_{\mu},A^{-\infty})=\{\mu_n: \ n\in\N_0\}.$
\end{proposition}

\begin{lemma}{\cite[Lemma 9]{Orts} }\label{LemmaOrts}
A function $f(z)=\sum_{n=0}^{\infty}a_nz^n\in H(\D)$ belongs to  $A^{-\infty}$
if and only if there exists $k\in \N$ such that $\sup_n\frac{|a_n|}{(n+k)^k}<\infty$.
\end{lemma}

\begin{proposition}
Let $\mu$ be a positive finite Borel measure on $[0,1)$ such that $\mu((0,1))>0$  and consider $C_\mu: A^{-\infty}\to A^{-\infty}.$ If $\mu$ is a Carleson measure, then the following are equivalent:
	\begin{itemize}
		\item[(i)] $0\notin \sigma(C_{\mu},A^{-\infty}).$
		\item[(ii)] 	$C_\mu: A^{-\infty}\to A^{-\infty}$  is surjective.
		\item[(iii)] $M_{(1/\mu_n)_n}: A^{-\infty}\to A^{-\infty},$ \ $f(z)=\sum_{n=0}^{\infty}b_nz^n\mapsto h_f(z)=\sum_{n=0}^{\infty}\frac{b_n}{\mu_n}z^n$ is continuous.
		\item[(iv)] There exist $C, s>0$ such that $\frac{1}{\mu_n}\leq Cn^s$ for every $n\in \N.$
	\end{itemize}
	
\end{proposition}

\begin{proof}
	(i) $\Leftrightarrow$ (ii) by Proposition \ref{Espectre Ainfty}.\\
	(ii) $\Leftrightarrow$ (iii) 	Proceeding as in the proof of Proposition \ref{C_mu sobre HD}, we get that
	$C_\mu: A^{-\infty}\to A^{-\infty}$  is surjective if and only if
	\begin{equation}\label{antiimatgeHD}
		g_f(z)=b_0+\sum_{n=1}^{\infty}\left( \frac{b_n}{\mu_n}-\frac{b_{n-1}}{\mu_{n-1}} \right)z^n\in A^{-\infty}
	\end{equation}
	for every $f(z)=\sum_{n=0}^{\infty}b_nz^n\in A^{-\infty}.$ And this condition is equivalent to  saying that the operator
	\begin{equation}\label{multipliers surj HD}
		M_{(1/\mu_n)_n}: A^{-\infty}\to A^{-\infty},\ f(z)=\sum_{n=0}^{\infty}b_nz^n\mapsto h_f(z)=\sum_{n=0}^{\infty}\frac{b_n}{\mu_n}z^n
	\end{equation}
	is continuous. Indeed,  $g_f(z)=(1-z)h_f(z)\in A^{-\infty}$ if and only if $h_f\in A^{-\infty}.$\\
	(iii) $\Leftrightarrow$ (iv)  Lemma \ref{LemmaOrts} implies that $h_f(z)=\sum_{n=0}^{\infty}\frac{b_n}{\mu_n}z^n\in A^{-\infty}$ if and only if there exist $k\in \N$  and $D>0$ such that $\frac{|b_n|}{\mu_n}\leq Dn^k$ for every $n\in \N.$ So, if this condition is satisfied and we consider  $f(z)=\frac{1}{1-z}\in A^{-\infty},$  since  $b_n=1$ for every $n\in \N$ we get $\frac{1}{\mu_n}\leq Dn^{k}$ for every $n\in \N.$ On the other hand, if  there exist $s, C>0$  such that $\frac{1}{\mu_n}\leq Cn^s$  for every $n\in \N,$ the conclusion follows easily  by Lemma \ref{LemmaOrts}, as $f\in A^{-\infty}$ yields the existence of $k\in \N$  and $D>0$ such that $b_n \leq Dn^k$ for every $n\in \N.$ Now,  Lemma \ref{LemmaOrts} implies (iv).
\end{proof}

\begin{example}
	Let $\mu$ be a positive finite Borel measure on $[0,1)$  and consider $C_\mu: A^{-\infty}\to A^{-\infty}.$
	\begin{itemize}
		\item[(i) ]	 If $\mu$ satisfies $\mu_n\cong \frac{1}{n^s},$ $s>0,$ then $0\notin\sigma(C_{\mu}).$
		\item[(ii)]  If $\mu$ satisfies  $\mu_n=\frac{1}{2^n},$ $n\in \N,$ then $0\in \sigma(C_{\mu}).$
	\end{itemize}
\end{example}

\begin{proposition}\label{Cmu-I sobre Korenb}
Let $\mu$ be a positive finite Borel measure on $[0,1)$ such that $\mu((0,1))>0$  and consider $C_\mu: A^{-\infty}\to A^{-\infty}.$ If $\mu$ is a Carleson measure and $\lambda\notin \{\mu_n: n\in \N_0 \}\cup\{0\},$  then  the operator $C_{\mu}-\lambda I:A^{-\infty}\to A^{-\infty} $ is surjective, i.e., $\lambda\notin \sigma(C_{\mu}, A^{-\infty})$.
\end{proposition}

\begin{proof}
Let $\lambda\notin \{\mu_n: n\in \N_0 \}\cup\{0\}.$ By the proof of Proposition \ref{C_mu-lambdaI sobre HD}, it is enough to show that for every $f(z)=\sum_{n=0}^{\infty}b_nz^n\in A^{-\infty},$  the functions $g_f(z)=\sum_{n=0}^{\infty}\frac{b_n}{\mu_n-\lambda}z^n\in A^{-\infty}$  and \begin{equation*}
		 h_f(z)=\sum_{n=0}^{\infty}\mu_n\left(\sum_{k=0}^{n-1}\frac{b_k}{\prod_{j=k}^{n}\left(1-\frac{\mu_j}{\lambda}\right)}\right)z^n\in A^{-\infty}.
	\end{equation*}
By Lemma \ref{LemmaOrts},  there exists $s\in \N$ such that $\sup_n\frac{|b_n|}{(n+s)^s}<\infty.$ Therefore, as there exists $\delta>0$ with $|\mu_n-\lambda|\geq \delta$ for every $n\in \N,$ we get $\sup_n\frac{|b_n|}{|\mu_n-\lambda|(n+s)^s}<\infty$ and $g_f\in A^{-\infty}.$
Now,  take $m\in \N, \ m\geq s + \lceil C|Re(1/\lambda)|\rceil,$ where $C>0$ is such that $\mu_n\leq \frac{C}{n}$ for every $n\in \N.$ Now, by Lemma \ref{Carleson_estim_inf_prod}(i) and Remark \ref{remarkl} we get

\begin{eqnarray*}
 \sup_n\frac{1}{(n+m)^m}\mu_n\left|\sum_{k=0}^{n-1}\frac{b_k}{\prod_{j=k}^{n}\left(1-\frac{\mu_j}{\lambda}\right)}\right|&\lesssim& \sup_n\frac{n^{\lceil C|Re(1/\lambda)|\rceil-1}}{(n+m)^m}\sum_{k=0}^{n-1}|b_k|\nonumber\\
 &\lesssim &\sup_n\frac{n^{\lceil C|Re(1/\lambda)|\rceil-1}}{(n+m)^m}(n+s-1)^{s+1}<\infty.
\end{eqnarray*}
Therefore, Lemma \ref{LemmaOrts} yields $h_f\in A^{-\infty}$.
\end{proof}

As a consequence of Propositions \ref{Espectre Ainfty} and \ref{Cmu-I sobre Korenb}, the following corollary holds:

\begin{corollary}
    Let $\mu$ be a positive finite Borel measure on $[0,1)$ such that $\mu((0,1))>0$ and consider $C_\mu: A^{-\infty}\to A^{-\infty}.$ If $\mu$ is a Carleson measure, then $\sigma(C_{\mu}, A^{-\infty})\setminus \{0\}=\{\mu_n: \ n\in\N_0\}.$
\end{corollary}



\subsection*{Acknowledgments}
The research of the  three authors  was partially supported by GVA-AICO/2021/170. The research of the second and third authors was also partially supported by the project PID2020-119457GB-100 funded by MCIN/AEI/10.13039/501100011033 and by
``ERFD A way of making Europe''.

\section*{Statements and Declarations}

\begin{itemize}
	\item \textbf{Funding:} 	The research of the  three authors  was partially supported by GVA-AICO/2021/170. The research of the second and third authors was also partially supported by the project PID2020-119457GB-100 funded by MCIN/AEI/10.13039/501100011033 and by
	``ERFD A way of making Europe''.
	
	\item 	\textbf{Competing Interests:} The authors have no relevant financial or non-financial interests to disclose.
	
		\item 	\textbf{Availability of data and materials:}  	 Data sharing not applicable to this article as no datasets were generated 	or analysed.
	
	\item 	\textbf{Authors' contributions:} All authors contributed equally to the study conception and design. All authors read and approved the final manuscript.
	
\end{itemize}


%
%
``ERFD A way of making Europe''.

%
%


\begin{appendices}
	
\section{Hausdorff moment problem for Borel probabilities on $[0,1)$} \label{appendix}

This appendix collects a few classical results about  Hausdorff moments  for Borel probabilities on $[0,1)$. They provide us with examples  which are relevant in connection with the theorems in the article.

 We start with the following result, which is known to specialists. We state it and give its short proof, since it is very useful in the article.

 \begin{lemma}\label{moments}
Let $\mu$ be a positive finite Borel measure on $[0,1)$   and $(\mu_n)_n$ its moments.Then the following hold:
\begin{itemize}
	\item[(i)] $(\mu_n)_n$ is a  decreasing sequence satisfying $\lim_n\mu_n=0.$
		\item[(ii)] $\left( \frac{\mu_n}{\mu_{n+1}}\right)_n$ is a decreasing sequence bounded below by $1$, so the limit $L=\lim_n\frac{\mu_n}{\mu_{n+1}}$ exists, with $L\geq 1.$
		\item[(iii)] If $\mu((0,1))>0$, then   $(\mu_n)_n$ is a strictly decreasing sequence.
\end{itemize}
\end{lemma}

\begin{proof}
 (i)  $(\mu_n)_n$ is  clearly   decreasing, so $\lim_n\mu_n=0$   by  the Monotone Convergence Theorem.\\
(ii)  By H\"older's inequality, we get
$$\int_{0}^{1} t^nd\mu\leq 	\left(\int_{0}^{1} (t^{\frac{n-1}{2}})^2d\mu\right)^{1/2} \left(\int_{0}^{1} (t^{\frac{n+1}{2}})^2d\mu\right)^{1/2},$$
therefore $\mu_n^2\leq \mu_{n-1} \mu_{n+1} $ and so, $\left( \frac{\mu_n}{\mu_{n+1}}\right)_n$ is a decreasing sequence. As $\frac{\mu_n}{\mu_{n+1}}\geq 1$ for every $n\in \N,$ then  $\lim_n\frac{\mu_n}{\mu_{n+1}}$ exists.\\
(iii)   By assumption, there are  $a,b\in (0,1)$  such that $\mu([a,b])>0.$ So, we can find $c_n>0$ such that $\int_{0}^{1}(t^n-t^{n+1})d\mu\geq\int_{a}^{b}(t^n-t^{n+1})d\mu\geq c_n\mu([a,b])>0, $ and we get $(\mu_n)_n$ is a strictly decreasing sequence.
\end{proof}

For a Borel probability measure on $[0,1)$, the condition $\mu((0,1))>0$ in Lemma \ref{moments} is equivalent to $\mu\neq \delta_0,$ where $\delta_0$ is the Dirac measure on $[0,1)$.\\

The \emph{difference sequences} of a given sequence $(a_n)_{n\in\N_0}$ of real numbers are defined by  $\Delta^0(a_n)=a_n$,  $\Delta^1(a_n)=a_{n+1}-a_n$, $\Delta^j(a_n)=\Delta^1(\Delta^{j-1}(a_n))$ for $j>1$. A sequence $(a_n)_n$ is said to be \emph{completely monotonic} whenever $((-1)^j\Delta^j(a_n))_{j\in\N_0}$ is a sequence of positive numbers for every $n\in \N$.

A sequence $(\mu_n)_{n\in\N_0}$ of positive numbers is said to be \emph{a solution of the Hausdorff moment problem} if there exists a Borel probability measure $\mu$ on $[0,1]$ such that $\mu_n=\int_0^1 t^nd\mu(t)$ for each $n\in\N_0$.

A proof of the following characterization due to Hausdorff can be seen in \cite{Schmudgen} Theorem 3.15.

\begin{theorem}[Hausdorff]
A sequence $(\mu_n)_{n\in\N_0}$ of positive numbers is a sequence of moments of a  Borel probability $\mu$ in $[0,1)$ if and only $\mu_0=1$, $(\mu_n)_{n\in\N_0}$ is completely monotonic and $\lim_n \mu_n=0$.
\end{theorem}

Hausdorff theorem is stated for measures on $[0,1]$ without the condition $\lim_n \mu_n=0$. If $\mu$ is a probability on $[0,1]$, the condition $\mu(\{1\})=0$ is easily seen to be equivalent to the condition $\lim_n \mu_n=0$ by  the Lebesgue dominated convergence theorem. Observe that $\int_{\{1\}} t^nd\mu=\mu(\{1\})$ for every $n\in\N_0$.

\begin{example}
If $0<a<1,$ then $(a^n)_{n\in\N_0}$ is the sequence of moments of a probability on $[0,1)$, as follows immediately from the Hausdorff theorem and the equality $(\Delta^{j}(a^n))_{j\in\N_0}=((a-1)^ja^n)_{j\in\N_0}$.
\end{example}

	A function $f\in C^\infty([0,\infty))$ is said to be \emph{completely monotonic} whenever $(-1)^jf^{(j)}(x)\geq 0$ whenever $j\in\N_0$.
The following theorem is a consequence of the Hausdorff theorem together with \cite[Theorem 11d, p.158]{Widder}.

\begin{theorem}	\label{absmon}
If $f\in  C^\infty([0,\infty))$ is a completely monotonic function such that $f(0)=1$  and $\lim_{x\to\infty}f(x)=0,$ then $(f(n))_{n=0}^{\infty}$ is a sequence of moment problems for a Borel probability on $[0,1)$.
\end{theorem}

\begin{example}
\label{exp}
Let $\mu_n=(n+1)^{-p}$, $p>0$.  The sequence $(\mu_n)_{n\in\N_0}$ is a  sequence of moments  for a Borel probability $\mu$ on $[0,1)$.
\end{example}

\begin{lemma}
\label{tec}
Assume that $f_1,f_2\in C^\infty[0,\infty)$ and $(-1)^jf_i^{(j)}(x)\geq 0$ for all  $x\in [0,\infty)$, $j\in\N_0$, $i=1,2$, $f_1(0)f_2(0)=1$,  $f_1$ bounded and $\lim_{x\to \infty}f_2(x)=0$. Then $((f_1(n)f_2(n))_{n=0}^\infty$ is a  sequence of moments for a Borel probability on $[0,1)$.
\end{lemma}

\begin{proof}
Under these hypothesis, $\lim_{x\to\infty}f_1(x)f_2(x)=0$. By applying the  Leibniz rule, we get
$$(-1)^n(f_1f_2)^{(n)}(x)=\sum_{j=0}^{n} \left ( \begin{array}{c}n\\ j\end{array}\right)(-1)^jf_1^{(j)}(x)(-1)^{n-j}f_2^{(n-j)}(x)\geq 0.$$

\noindent The result follows now immediately from Theorem \ref{absmon}.
\end{proof}

The following  example shows that, beyond multiples of the Cesàro operator, there are more examples satisfying the conditions of Proposition \ref{Espectre puntual Hv}.

\begin{example}\label{Ex:nomultipleCesaro}
Let $\mu_n=\frac{n+2}{2(n+1)^2}$, $n\in \N_0.$ There exists a probability $\mu$ on $[0,1)$ which is the  solution of the Hausdorff moment problem for $(\mu_n)_{n\in \N_0}$. Indeed, it is enough to observe that, if one takes $f_1(x)=1+\frac{1}{x+1}$ and $f_2(x)=\frac{1}{2(1+x)}$, then $\mu_n=f_1(n)f_2(n)$. The conclusion holds by Lemma \ref{tec}.
\end{example}

\end{appendices}

{\bf Acknowledgments}{The research of the  three authors  was partially supported by GVA-AICO/2021/170. The research of the second and third authors was also partially supported by the project PID2020-119457GB-100 funded by MCIN/AEI/10.13039/501100011033 and by
	``ERFD A way of making Europe''.}


\begin{thebibliography}{}
		
				
		\bibitem{ABR2016} Albanese, A.A., Bonet, J., Ricker, W.J.: The Cesàro Operator in Growth Banach Spaces of Analytic Functions. Integr. Equ. Oper. Theory 86, 97--112 (2016)
		
		\bibitem{ABR2018} Albanese, A.A., Bonet, J.,  Ricker, W.J.:  The Cesàro operator on Korenblum type spaces of analytic functions. Collect. Math. 69(2), 263–-281 (2018)
		
		\bibitem{ABR2019} Albanese, A.A., Bonet, J., Ricker, W.J.: Operators on the Fréchet sequence space $ces(p+),$ $1 \leq p < \infty$. Revista de la Real Academia de Ciencias Exactas Físicas y Naturales Serie A Matemáticas RACSAM 113(2), 1533--1556 (2019)
				
		\bibitem{AlemanPersson2010} Aleman, A., Persson, A.M.: Resolvent estimates and decomposable extensions of generalized Cesàro operators. J. Funct. Anal. 258, 67--98 (2010)

\bibitem{Arroussi} Arroussi, H.: Function and operator theory on large Bergman spaces. Ph.D. Thesis. Universitat de Barcelona (2016)
		
		\bibitem{BaoSunWulan} Bao, G., Sun, F. and Wulan, H.: Carleson measures and the range of a Cesàro-like operator acting on $H^{\infty}$. Anal. Math. Phys. 12, 142 (2022)
		
		\bibitem{BBG} Bierstedt, K.D., Bonet, J., Galbis, A.: Weighted spaces of holomorphic functions on balanced domains. Michigan Math. J. 40(2), 271--297 (1993)
		
		\bibitem{BBT} Bierstedt, K.D., Bonet, J., Taskinen, J.: Associated weights and spaces of holomorphic functions. Studia Math. 127(2), 137–-168 (1998)
		
		\bibitem{BS} Bierstedt, K.D.,  and  Summers, W.H.:  Biduals of weighted Banach spaces of analytic functions. J. Austral. Math. Soc. Ser. A 54, 70-79 (1993)
		
		\bibitem{Blasco} Blasco, O.: Cesàro-type operators on Hardy spaces. J. Math. Anal. Appl. (2023), https://doi.org/10.1016/j.jmaa.2023.127017
	
			\bibitem{pepesurvey} Bonet, J.:  Weighted Banach spaces of analytic functions with sup-norms and operators between them: a survey. Revista de la Real Academia de Ciencias Exactas Físicas y Naturales Serie A Matemáticas RACSAM 4 (116), 1--40 (2022).
		
		\bibitem{CGP} Chatzifountas,   Ch., Girela, D.,   Peláez, J.A.: A generalized Hilbert matrix acting on Hardy spaces. J. Math. Anal. Appl. 413(1), 154–168 (2014)
		
		\bibitem{Conway} Conway J.: A Course in Functional Analysis. Second Edition. Springer (2007)

  \bibitem{Dahlner} Dahlner,  A.: Some resolvent estimates in harmonic analysis; decomposable extension of the Cesàro operator on the weighted Bergman space
and Bishop’s property. Doctoral thesis in Mathematical Sciences, 2003:4, Centre for Mathematical Sciences, Lund University (2003)

  \bibitem{DanikasSiskakis} Danikas, N., Siskakis, A.G.: The Cesàro operator on bounded analytic functions. Analysis 13(3), 295–299 (1993)
	
		
	\bibitem{Dostanic}	Dostanic, M.R.:  Integration Operators on Bergman Spaces with exponential weight. Revista Matemática Iberoamericana 23(2), 421--436 (2007)
	
	\bibitem{GGMM} Galanopoulos, P., Girela, D., Mas, A.,  Merchán, N.: Operators induced by radial measures acting on the Dirichlet space. Results Math. 78:106 (2023)
	
		
		\bibitem{GGM} Galanopoulos, P., Girela, D.,  Merchán, N.: Cesàro-like operators acting on spaces of analytic functions. Anal. Math. Phys. 12(2), 51 (2022)
		
				\bibitem{GGM2023} Galanopoulos, P., Girela, D.,  Merchán, N.: Cesàro-type operators associated with Borel measures on the unit disc acting on some Hilbert spaces of analytic functions. J. Math. Anal. App. 526(2), 127287(2023)
				
			
		  \bibitem{Orts} Gómez-Orts, E.: Weighted composition operators on Korenblum type spaces of analytic functions. Revista de la Real Academia de Ciencias Exactas Físicas y Naturales Serie A Matemáticas RACSAM 114, 199 (2020)
		
		  	\bibitem{Knopp} Knopp, K.: Infinite Sequences and Series. Dover, New York (1956).
		
		  		\bibitem{Korenblum} Hedenmalm, H., Korenblum, B., Zhu, K.: Theory of Bergman Spaces. Grad. Texts in Math., vol. 199. Springer, New York (2000)
		
		\bibitem{Jarchow} Jarchow, H.: Locally convex spaces, B. G. Teubner, Stuttgart, 1981. Mathematische Leitf\"aden. [Mathematical Textbooks]. MR632257 (83h:46008)
		
		
		\bibitem{Lusky2006} Lusky, W.: On the isomorphism classes of weighted spaces of harmonic and holomorphic functions. Studia Math. 175(1), 19--40 (2006)
		
		\bibitem{meisevogt} Meise, R., Vogt, D.: Introduction to Functional Analysis. Clarendon Press, Oxford University Press, New York (1997)


   \bibitem{Alex} Miralles, A.: Schur spaces and weighted spaces of type $H^{\infty}$. Quaest. Math., 35(4), 463-470 (2012)
		
		\bibitem{Persson2008} Persson, A. M.: On the spectrum of the Cesàro operator on spaces of analytic functions. J. Math. Anal. Appl. 340, 1180–1203 (2008)

\bibitem{Schmudgen} Schm\"udgen, K.: The Moment problem. Springer Verlag (2017)
		
		\bibitem{Shapiro} Shapiro, J.H.: Composition Operators and Classical Function Theory. Springer Verlag (1993)
		
		\bibitem{SW1971} Shields, A.L., Williams, D.L.: Bounded projections, duality and multipliers in spaces of analytic functions. Trans. Am. Math. Soc. 162, 287--302 (1971)
		
		\bibitem{Widder} Widder, D. V.: The  Laplace  transform. Princeton  University  Press,  Princeton (1941).
		
		
	\end{thebibliography}
\end{document}